\documentclass[a4paper,reqno, 10pt]{amsart}

\usepackage{amsmath,amssymb,amsfonts,amsthm, mathrsfs}
\usepackage[latin1]{inputenc}
\usepackage{lmodern}
\usepackage{tikz}

\usepackage{verbatim,wasysym,cite}

\usepackage[latin1]{inputenc}
\usepackage{microtype}
\usepackage{color,enumitem,graphicx}
\usepackage[colorlinks=true,urlcolor=blue, citecolor=red,linkcolor=blue,linktocpage,pdfpagelabels, bookmarksnumbered,bookmarksopen]{hyperref}

\numberwithin{equation}{section}
\newtheorem{theorem}{Theorem}[section]
\newtheorem{lemma}[theorem]{Lemma}
\newtheorem{remark}[theorem]{Remark}
\newtheorem{definition}[theorem]{Definition}
\newtheorem{proposition}[theorem]{Proposition}
\newtheorem{corollary}[theorem]{Corollary}

\newcommand{\qtq}[1]{\quad\text{#1}\quad}

\newcommand{\R}{\mathbb{R}}
\newcommand{\C}{\mathbb{C}}

\newcommand{\eps}{\varepsilon}

\numberwithin{equation}{section}
\numberwithin{theorem}{section}
\numberwithin{figure}{section}
\numberwithin{table}{section}

\def\({\left(}
\def\){\right)}
\def\<{\left\langle}
\def\>{\right\rangle}

\def\eps{\varepsilon}

\let\Re=\undefined\DeclareMathOperator*{\Re}{Re}
\let\Im=\undefined\DeclareMathOperator*{\Im}{Im}

\begin{document}
	
\title[4D mass-energy double critical NLS]{Threshold dynamics for the 4$d$ mass-energy double critical NLS}

\author[Alex H. Ardila]{Alex H. Ardila}
\address{Alex H. Ardila
\newline \indent Department of Mathematics, Universidad del Valle, Colombia} 
\email{ardila@impa.br}

\author[Zuyu Ma]{Zuyu Ma}
\address{Zuyu Ma
\newline \indent  The Graduate School of China Academy of Engineering Physics, Beijing 100088,\  China}
\email{mazuyu23@gscaep.ac.cn}
	
\author[Jason Murphy]{Jason Murphy}
\address{Jason Murphy
\newline \indent  Department of Mathematics, University of Oregon, Eugene, OR 97403,  USA}
\email{jamu@uoregon.edu}
	
\author[Jiqiang Zheng]{Jiqiang Zheng}
\address{Jiqiang Zheng
\newline \indent  Institute of Applied Physics and Computational Mathematics,  \ Beijing, \ 100088, \ China}
\email{zheng\_jiqiang@iapcm.ac.cn, zhengjiqiang@gmail.com}
	

\begin{abstract}
We consider the $4d$ mass-energy double critical NLS
\[
(i\partial_t+\Delta)u = -|u|^2 u + |u| u.
\]
In \cite{Luo2024, ChengMiaoZhao2016}, the authors established a scattering/blowup dichotomy for solutions satisfying the energy constraint $E(u_0)< E^c(W)$, where $W$ is the energy-critical NLS ground state and $E^c$ is the energy for the underlying cubic NLS.  We prove that the scattering/blowup dichotomy persists even at the energy threshold $E(u_0)=E^c(W)$.   
\end{abstract}

\maketitle

\section{Introduction}\label{sec:intro}
We consider the nonlinear Schr\"odinger equation with combined nonlinearity of the form
\begin{equation}\label{NLS}\tag{DCNLS}
\begin{cases}
(i\partial_{t}+\Delta)u=-|u|^{2}u+|u|u,\\
u|_{t=0}=u_{0}\in H^{1}(\R^{4}),
\end{cases}
\end{equation}
with $u:\R\times\R^4\to\C$. This is the Hamiltonian flow corresponding to the following conserved \emph{energy}, which is finite for $H^1$ data: 
\[
E(u)=\int_{\R^{4}}\Big(\tfrac{1}{2}|\nabla u|^{2}-\tfrac{1}{4}|u|^{4}+\tfrac{1}{3}|u|^{3}\Big)\,dx.
\]
Solutions additionally conserve the \emph{mass}, defined by
\[
M(u)=\int_{\R^4} |u|^2\,dx.
\]

Equation \eqref{NLS} models physical systems such as pulse propagation in non-centrosymmetric optical waveguides and the axial dynamics of cigar-shaped Bose-Einstein condensates, where it describes the interplay between repulsive contact interactions (cubic nonlinearity) and dipole-dipole attractions (quadratic nonlinearity). Furthermore, the nuclear hydrodynamics equation with effective Skyrme forces is known to reduce quasi-classically to the \eqref{NLS} form. For further physical context, we refer the reader to \cite{MateoDelgado2008, SinhaSantos2007} and references therein.

An important related model is the $4d$ focusing energy-critical NLS
\begin{equation}\label{ECgNLS}\tag{NLS}
\begin{cases} 
(i\partial_{t}+\Delta) u=-|u|^{2}u,\\
u|_{t=0}=u_{0}\in \dot{H}^{1}(\R^{4}),
\end{cases} 
\end{equation}
with associated energy
\[
E^{c}(u)=\int_{\R^{4}}\Big(\tfrac{1}{2}|\nabla u|^{2}-\tfrac{1}{4}|u|^{4}\Big)\,dx.
\]
	
The equation \eqref{ECgNLS} has static solutions of the form $u(t,x)=W(x)$, where $W$ solves the elliptic equation
\begin{align}\label{ellipC}
-\Delta W=|W|^{2}W.
\end{align}
An explicit solution, known as the \emph{ground state}, is given by 
\[
W(x):= \left(1+\tfrac{|x|^{2}}{8}\right)^{-1}.
\]

In recent years, the study of dynamics for NLS with combined nonlinearities has attracted significant attention. A systematic study was initiated by Tao, Vi\c{s}an and Zhang in the work \cite{TaoVisanZhang2007}. For the NLS with defocusing-defocusing double critical regime, they showed (using the interaction Morawetz estimate) that solutions are global and scatter in time for any initial data in $H^1$. Subsequently, Akahori, Ibrahim, Kikuchi and Nawa~\cite{AkaIbraKikuNawa2013} formulated a sharp scattering threshold for the case of focusing energy-critical NLS perturbed by a focusing intercritical nonlinearity, and studied the solution dynamics below this threshold. Let us also mention the works \cite{MiaoChangxing2013, KillipOhPoVi2017, Luo2022, MiaoChangxing2015,MiaoChangxing2017}, which have studied the global dynamics of NLS with combined nonlinearities in different regimes.

For equation \eqref{NLS}, the behavior of solutions satisfying the energy constraint $E(u)<E^c(W)$ is now well-understood.  This problem was first considered in \cite{ChengMiaoZhao2016} in the radial setting, with the radial assumption later being removed in \cite{Luo2024} using the so-called double track profile decomposition. 

\begin{theorem}[Sub-threshold dynamics, \cite{Luo2024, ChengMiaoZhao2016 }]\label{Th1}
Let $u_{0}\in H^{1}(\R^{4})$, and let $u$ be the corresponding maximal-lifespan solution to \eqref{NLS} with $u|_{t=0}=u_{0}$.
\begin{enumerate}[label=\rm{(\roman*)}]
\item $($Subcritical$)$ If $u_{0}$ obeys
\begin{equation}\label{Thres00}
E(u_{0})<E^{c}(W) \quad \text{and} \quad \|\nabla u_{0}\|_{L^{2}}<\|\nabla W\|_{L^{2}},
\end{equation}
then $u$ is global, $u\in L_{t,x}^{6}(\R\times\R^4)\cap L_{t,x}^{3}(\R\times\R^4)$, and scatters\footnote{Here \emph{scattering} as $t\to\pm\infty$ refers to the statement that
\[
\lim_{t\to\pm\infty}\|u(t)-e^{it\Delta}u_\pm\|_{H^1}=0\qtq{for some}u_\pm\in H^1.
\]}	 as $t\to\pm\infty$. 
\item$($Supercritical$)$ If $u_{0}$ is radial and obeys
\begin{equation}\label{BlowC00}
E(u_{0})<E^{c}(W) \quad \text{and} \quad \|\nabla u_{0}\|_{L^{2}}>\|\nabla W\|_{L^{2}},
\end{equation}
then $u$ blows up in both time directions in finite time.
\end{enumerate}
\end{theorem}	

Our goal is to study the dynamics of \eqref{NLS} at the energy threshold $E(u_0)=E^c(W)$. In particular, we prove that scattering in the subcritical case continues to hold even at the threshold energy. 
	
\begin{theorem}[Threshold dynamics]\label{Th2}  Let $u_{0}\in H^{1}(\R^4)$  and let $u$ be the corresponding maximal-lifespan solution to \eqref{NLS} with $u|_{t=0}=u_0$. If $u_{0}$ obeys
\begin{equation}\label{Thres}
E(u_{0})=E^{c}(W) \quad \text{and} \quad \|\nabla u_{0}\|_{L^{2}}<\|\nabla W\|_{L^{2}},
\end{equation}
then $u$ is global, $u\in L_{t,x}^{6}(\R\times\R^4)\cap L_{t,x}^{3}(\R\times\R^4)$, and scatters as $t\to\pm\infty$.
\end{theorem}

By applying essentially the same argument developed in \cite[Section 4]{AJZ}, we can also show that if a solution $u$ to \eqref{NLS} with radial initial data $u_0$ satisfies the condition 
\begin{equation}\label{ThresBlo}
E(u_0) = E^c(W) \quad \text{and} \quad \|\nabla u_0\|_{L^2} > \|\nabla W\|_{L^2},
\end{equation}
then the solution $u$ blows up in both time directions in finite time. We also expect that an analogous result holds for the equation
\begin{equation*}
\begin{cases}
(i\partial_{t} + \Delta) u = -|u|^{\frac{4}{d-2}}u + |u|^{\frac{4}{d}}u,\\
u|_{t=0}=u_{0} \in H^{1}(\mathbb{R}^{d})
\end{cases}
\end{equation*}
in dimensions $d\geq 5$. 

Our main result fits in the context of ground state threshold dynamics for nonlinear dispersive PDE, a topic of intense ongoing research. For the case NLS with combined nonlinearities, we specifically mention the recent works \cite{AJZ, KMV, Murphy, ArdilaMurphy2023, HamanoKikuchiWatanabe2023}. In particular, in \cite{AJZ} the dynamics at the energy threshold is classified in three space dimensions for the focusing energy-critical NLS perturbed by a defocusing cubic nonlinearity with radial initial data. For the case of the $3d$ defocusing energy-critical NLS perturbed by a focusing cubic nonlinearity (cubic-quintic NLS), the dynamics were classified in \cite{ArdilaMurphy2023} without assuming radiality of the initial data. 

Finally, we expect that the methodology developed in this paper is fully applicable to the study of threshold \emph{non-radial} scattering for the focusing energy-critical NLS perturbed by a general defocusing energy-intercritical nonlinearity
\begin{equation*}
\begin{cases}
(i\partial_{t} + \Delta) u = -|u|^{\frac{4}{d-2}}u + |u|^{p}u,\quad  p\in(\frac{4}{d},\frac{4}{d-2})\\
u|_{t=0}=u_{0} \in H^{1}(\mathbb{R}^{d})
\end{cases}
\end{equation*}
in dimensions $d \geq 4$. In particular, the proof will be greatly simplified because we do not need to establish compactness (modulo the symmetry group) in $L^2$ for the blowup solution, which is necessary and crucial in our double critical setting, as discussed in detail below.

\subsection{Strategy of the proof, part I} 

The proof of Theorem~\ref{Th2} proceeds by contradiction.  In particular, if Theorem~\ref{Th2} fails then there exists a non-scattering solution $u$ satisfying the energy constraints \eqref{Thres}. Without loss of generality, we may assume that $u$ fails to scatter forward in time.  Using the double-track profile decomposition from \cite{Luo2024} and the sub-threshold scattering result Theorem~\ref{Th1}, we can prove that the solution $u$ must be precompact in $\dot H^1$ modulo a time-dependent frequency scale $\lambda(t)\geq 1$ and spatial center $x(t)$. We can additionally prove that $u$ is forward global.  

To complete the proof, we would like to show that such a solution cannot exist.  In general, this requires an analysis of the compactness parameters $\lambda(t)$ and $x(t)$.  For example, in the radial setting one has $x(t)\equiv 0$. This simplifies the problem considerably, as it allows for the direct use of a virial-type estimate (cf. \cite{AJZ}).  Unfortunately, in the non-radial setting we have no \emph{a priori} control over $x(t)$ in general.  Nonetheless, there is one situation in which we can effectively control $x(t)$, namely, when the solution $u$ is very close to the orbit of the ground state (we call this the \emph{small modulation regime}).  To make this precise, we define
\begin{equation}\label{deltau}
\delta(u(t)):=\|\nabla W\|_{L^2}^2 - \|\nabla u(t)\|_{L^2}^2
\end{equation}
and observe that if $0<\delta(u(t))\ll 1$, we can define modulation parameters $(\theta(t),\mu(t),x_0(t))$ yielding a decomposition
\[
e^{i\theta(t)}\tfrac{1}{\mu(t)} u(t,\tfrac{x-x_0(t)}{\mu(t)})=[1+\alpha(t)]W + h(t)
\]
for suitable small $\alpha,h$.  The time evolutions of the modulation parameters are dictated by certain ODEs, which in particular allow us to obtain \emph{a priori} control over $x_0(t)$.  

In fact, this motivates our fundamental strategy, similar to that appearing in works such as \cite{MaMiaoMurphyZheng2025, DodsonMC}: \emph{we seek to reduce the problem to the preclusion of an almost periodic solution belonging to the small modulation regime for all (forward) time.} Indeed, using a bootstrap argument based on modulation analysis and (modulated) virial estimates (inspired by \cite{DodsonMC}), we can ultimately preclude the existence of such solutions and thereby complete the proof of Theorem~\ref{Th2}.

In fact, while this overall approach is similar to that of \cite{MaMiaoMurphyZheng2025, DodsonMC} (which address the mass and energy critical equations individually), carrying out the reduction described above leads to significant new challenges in our setting due to the double critical nonlinearity.  At a technical level, the difficulty arises from the fact that while we naturally obtain compactness for blowup solutions at the $\dot H^1$ level, carrying out the reduction described above potentially requires compactness in $L^2$ level, as well.  This leads us to consider a new formulation of the problem, which in particular involves the notion of \emph{minimal mass for blowup} at the energy threshold.

\subsection{Strategy of the proof, part II}\label{S:strategy} We now outline in more detail our approach to proving Theorem~\ref{Th2}.

\begin{definition}[Threshold blowup set]\label{D:B} Let us define $\mathcal{B}$ to be the set of maximal-lifespan solutions $u:I_{\max}\times\R^4\to\C$ to \eqref{NLS} with $u(0)\in H^1$ that satisfy the energy constraint \eqref{Thres} and blow up forward in time.
\end{definition}

We define the \emph{critical mass}
\begin{equation}\label{m_c}
m_c := \inf_{u\in\mathcal{B}} M(u) \in[0,\infty] 
\end{equation}
and introduce a universal small constant $\eta_*>0$ (quantifying the notion of sufficiently small modulation) that will be determined in the analysis below. 

In light of time-reversal symmetry, Theorem~\ref{Th2} is equivalent to proving that $\mathcal{B}=\emptyset$.  We argue by contradiction and therefore assume throughout the paper that $\mathcal{B}\neq \emptyset$.  Note that this assumption implies $m_c\in[0,\infty)$. 

The proof of Theorem~\ref{Th2} can then be broken into several propositions, whose proofs in Sections~\ref{S:Compactness}--\ref{S:preclude} comprise the heart of this paper.

We first establish the following $\dot H^1$-compactness property for threshold blowup solutions. 
\begin{proposition}[Blowup solutions are compact]\label{P2} If $u\in \mathcal{B}$, then there exist $\lambda:[0,\sup I_{\max})\to[1,\infty)$ and $x:[0,\sup I_{\max})\to\R^4$ such that
\begin{equation}\label{CompactX}
\bigl\{\tfrac{1}{\lambda(t)}u(t, \tfrac{x-x(t)}{\lambda(t)}): t\in [0, \sup I_{\max})\bigr\} \quad
\text{is precompact in $\dot{H}^{1}(\R^4)$}.
\end{equation}
\end{proposition}

With Proposition~\ref{P2} in hand, we can then establish global existence for solutions obeying the energy constraint. 

\begin{proposition}[Global existence]\label{P3} If $u:I_{\max}\times\R^4\to\C$ is a maximal-lifespan solution to \eqref{NLS} satisfying $u(0)\in H^1$ and the energy constraint \eqref{Thres}, then $I_{\max}=\R$. 
\end{proposition}

We next establish the following property relating frequency scale and modulation. 
\begin{proposition}[Scale vs. modulation]\label{P4} If $u\in\mathcal{B}$ and $t_n\to\infty$ is such that $\lambda(t_n)\to\infty$, then $\delta(u(t_n))\to 0$. 
\end{proposition}

We then take the following important step and prove that any threshold blowup solution must enter the small modulation regime along some time sequence $t_n\to\infty$. Here we argue by contradiction and adapt the interaction Morawetz estimate of \cite{D5} (which is in fact simplified in our setting due to the fact that $u\in L^2$). Importantly, this estimate is insensitive to the behavior of $x(t)$.  

\begin{proposition}[Small modulation along a sequence]\label{P5} If $u\in\mathcal{B}$, then there exists a sequence $t_n\to\infty$ such that $\delta(u(t_n))\to 0$. 
\end{proposition}

We next upgrade Proposition~\ref{P5} and prove that if $\mathcal{B}\neq \emptyset$, it must contain an element that belongs to the small modulation regime for its entire (forward) lifespan.  Our approach to this step is one of the key novelties of this work.  It is in this step that we utilize the notion of the minimal mass for blowup threshold solutions. The proof is carried out in Section~\ref{S:SMAT}. 

\begin{proposition}[Small modulation for all time]\label{P6} If $\mathcal{B}\neq \emptyset$, then there exists $u\in\mathcal{B}$ such that
\begin{equation}\label{SMAT}
\sup_{t\in[0,\infty)}\delta(u(t))\leq\eta_*.
\end{equation}
\end{proposition}

Finally, we show that the type of solution obtained in Proposition~\ref{P6} cannot exist.

\begin{proposition}[Preclusion of small modulation solutions]\label{P7} There is no $u\in\mathcal{B}$ satisfying \eqref{SMAT}. 
\end{proposition}

The primary tool for proving Proposition~\ref{P7} is the modulated virial estimate, which is a version of the localized virial estimate that takes into account the fact that the solution stays near the orbit of the ground state.  The presence of the mass-critical nonlinear term leads to the consideration of an additional term of the form
\[
A(t) = \int \Delta[w_R] |u(t,x)|^3\,dx,
\]
where $w_R$ is the virial weight. However, the modulation analysis allows us to estimate this term as $\mathcal{O}(\delta(u(t)))^2$, which allows us to treat this term perturbatively.

Using the modulated virial estimate, we can prove that the modulation parameters of the solution are nearly constant.  However, this turns out to be inconsistent with the fact that $\delta(u(t_n))\to 0$ along some sequence $t_n\to\infty$ (cf. Proposition~\ref{P5}) and therefore leads to a contradiction in our setting. The proof of this fact relies on the presence of the mass-critical perturbation.  In the energy-critical case, for example, the near-constancy of the modulation parameters instead leads to the exponential convergence of threshold solutions to the orbit of the ground state (cf. \cite{DuyckaMerle2009,MaMiaoMurphyZheng2025}).

Finally, Proposition~\ref{P6} and Proposition~\ref{P7} together imply that $\mathcal{B}=\emptyset$ and hence we conclude the proof of Theorem~\ref{Th2}.

\subsection{Organization of the paper}  After introducing some notation below, we import various preliminary results in Section~\ref{S:preli}, including variational analysis of the ground state, well-posedness and stability for \eqref{NLS}, the double-track linear profile decomposition, and some approximation results for \eqref{NLS}. Sections~\ref{S:Compactness}--\ref{S:preclude} contain the proofs of Propositions~\ref{P2}--\ref{P7} described above, which collectively prove our main result, Theorem~\ref{Th2}.  In Appendix~\ref{S:Modulation} we provide the details of the modulation analysis, and in Appendix~\ref{sec:App} we prove a result demonstrating that there can be no uniform estimate on the spacetime norms of solutions as one approaches the sharp energy threshold (see Theorem~\ref{T:nounif}).

\subsection*{Notation.}
Given positive quantities $A$ and $B$, we use $A\lesssim B$ or $B\gtrsim A$ to denote $A\leq CB$ for some positive constant $C>0$. We also use $A\sim B$ to indicate  $A \lesssim B \lesssim A$.  For $u:I\times \R^4\rightarrow \C$, $I\subset \R$, we write
\[ \|  u \|_{L_{t}^{q}L^{r}_{x}(I\times \R^4)}=\left\|  \|u(t) \|_{L^{r}_{x}(\R^4)} \right \|_{L^{q}_{t}(I)}
\]
with $1\leq q\leq r\leq\infty$. We recall that a pair of exponents $(q, r)$ is {admissible} if $2 \leq q, r \leq \infty$ and $\tfrac{2}{q} + \tfrac{4}{r} = 2$. Given a spacetime slab $I \times \mathbb{R}^4$, we define
\[
\|u\|_{S(I)} := \sup \left\{ \|u\|_{L^q_t L^r_x(I \times \mathbb{R}^4)} : (q, r) \text{ is admissible} \right\}.
\]
We let $\<\nabla\>=(1-\Delta)^{1/2}$ and we define the Sobolev norms 
\[
\|u\|_{H^{s,r}(\R^4)}:=\|\langle\nabla\rangle^{s} u\|_{L^{r}_{x}(\R^4)}.
\]
        
We will use the expression $X\pm$ to denote $X\pm\eps$ for some small $\eps>0$. We will make use of the $L^2$ and $\dot H^1$ inner products given by $(f,g)_{L^2}=\Re\int f\bar g$ and $(f,g)_{\dot H^1}=(\nabla f,\nabla g)_{L^2}$. We will use $A^\perp$ to denote the orthogonal complement of a set $A$. 
 
Moreover, for $f\in H^{1}(\R^4)$ we denote
\[
\delta(f):=\|\nabla W\|_{{L}^{2}}^{2}-\|\nabla f\|_{{L}^{2}}^{2}.
\]
		
\subsection*{Acknowledgements} A. Ardila was supported by the Universidad del Valle under research project CI-71425. J. Murphy
was partially supported by NSF grant DMS-2350225 and Simons Foundation grant MPS-TSM-00006622. J. Zheng was supported by National Key R$\&$D program of China: 2021YFA1002500 and NSF grant of China (No. 12271051).

\section{Preliminaries}\label{S:preli}
		
We will need some Littlewood--Paley theory. We let $\psi\in C_c^\infty(\R^4)$ be nonnegative and radial, with $\psi(x)=1$ if $|x|\leq 1$  and $\psi(x)=0$ if $|x|\geq \frac{11}{10}$. For $N\in 2^{\mathbb{N}}$, we define the Littlewood-Paley projections as Fourier multiplier operators, namely:  
\[
\widehat{P_{\leq N}f}(\xi):=\psi(\tfrac{\xi}{N})\hat{f}(\xi), \quad \widehat{P_{> N}f}(\xi):=\bigl[1-\psi(\tfrac{\xi}{N})\bigr]\hat{f}(\xi),
\]
and
\[
\widehat{P_{N}f}(\xi):=\bigl[\psi(\tfrac{\xi}{N})-\psi(\tfrac{2\xi}{N})\bigr]\hat{f}(\xi).
\]
These operators are bounded on Sobolev spaces and obey the following standard estimates.
\begin{lemma}[Bernstein inequalities]\label{BIN}
Fix $1\leq p\leq q\leq \infty$ and $s\geq0$. For $f: \R^4\to \C$, we have
\begin{align*}
\| P_{N}f   \|_{L^{q}(\R^4)}&\lesssim N^{\frac{4}{p}-\frac{4}{q}}\| P_{N}f   \|_{L^{p}(\R^4)},\\
\| P_{> N}f   \|_{L^{p}(\R^4)}&\lesssim N^{-s}\| |\nabla|^{s} P_{> N}f   \|_{L^{p}(\R^4)}.
\end{align*}
\end{lemma}
		
\subsection{Variational analysis} We recall some well-known properties of the ground state $W\in\dot H^1$, given by the explicit formula $W(x)=(1+\frac{1}{8}|x|^{2})^{-1}$.
		
The sharp Sobolev inequality takes the form
\begin{equation}\label{GI}
\|f\|_{L^{4}(\R^4)}\leq C_{GN}\|\nabla f\|_{L^{2}(\R^4)},
\end{equation}
with 
\begin{equation}\label{C_GN}
C_{GN}=\frac{\|W\|_{L^{4}}}{\|\nabla W\|_{L^{2}}}.
\end{equation}
We have the following variational characterization of the ground state: If $f$ is a nonzero function satisfying $\|f\|_{L^{4}}= C_{GN}\|\nabla f\|_{L^{2}}$, then 
\[
f(x)=e^{i\theta}\lambda W(\lambda(x-x_{0}))
\]
for some $\theta\in\R$, $\lambda>0$ and $x_{0}\in \R^4$.
		
The ground state $W$ also satisfies the following Pohozaev identities:
\begin{equation}\label{PoQ}
\|\nabla W\|^{2}_{L^{2}}=\| W\|^{4}_{L^{4}} \qtq{and} E^{c}(W)=\tfrac{1}{4}\|\nabla W\|^{2}_{L^{2}}.
\end{equation}
		
We also remark that for $f\in\dot H^1$ satisfying $\|\nabla  f\|_{L^{2}}\leq \|\nabla W\|_{L^{2}}$, we have
\begin{align}\label{IneW}
\tfrac{\|\nabla  f\|^{2}_{L^{2}}}{\|\nabla W\|^{2}_{L^{2}}}\leq \tfrac{E^{c}(f)}{E^{c}(W)}
\end{align}
(see \cite[Claim 2.6]{DuyckaMerle2009}).
		
We next show that the energy constraint in \eqref{Thres} persists in time. 

\begin{lemma}\label{GlobalS}
Assume that $u_{0}\in H^{1}(\R^4)$ satisfies
\begin{equation}\label{Condition11}
E(u_{0})=E^{c}(W)\quad \text{and}\quad \|\nabla u_{0}\|_{L^{2}}<\|\nabla W\|_{L^{2}}.
\end{equation}
Then the solution $u$ to \eqref{NLS} with $u|_{t=0}=u_0$ satisfies 
\begin{equation}\label{PositiveP}
\|\nabla u(t)\|_{L^{2}}<\|\nabla W\|_{L^{2}}
\end{equation}
throughout its maximal lifespan.  
\end{lemma}
		
\begin{proof} Suppose that there exists $t_{0}$ such  that $\|\nabla u(t_{0})\|^{2}_{L^{2}}=\|\nabla W\|^{2}_{L^{2}}$. By \eqref{IneW} and \eqref{Condition11}, we obtain the contradiction
\[
1=\frac{\|\nabla  u(t_{0})\|^{2}_{L^{2}}}{\|\nabla W\|^{2}_{L^{2}}}\leq \frac{E^{c}(u(t_{0}))}{E^{c}(W)}<\frac{E(u(t_{0}))}{E^{c}(W)}=1.
\]
\end{proof}
		
\subsection{Virial identities}\label{S:virial} Let $\phi$ be a smooth radial function satisfying
\[
\phi(x)=\begin{cases} |x|^{2},& \quad |x|\leq 1\\ 0,& \quad |x|\geq 2,
\end{cases}
\quad \text{with}\quad |\partial^{\alpha}\phi(x)|\lesssim |x|^{2-|\alpha|}
\]
for all multiindices $\alpha$.  Given $R>1$ we define
\begin{equation}\label{WRR}
w_{R}(x)=R^{2}\phi\left(\tfrac{x}{R}\right)
\end{equation}
and introduce the localized virial functional
\begin{equation}\label{I_R}
I_{R}[u]=2\Im\int_{\R^4} \nabla w_{R} \cdot\overline{u} \nabla u \,dx.
\end{equation}
		
By direct calculation, we have the following virial identity:
		
\begin{lemma}\label{VirialIden}
Let $R\in [1, \infty)$. Assume that $u$ solves \eqref{NLS}. Then we have
\begin{equation}\label{LocalVirial}
\tfrac{d}{d t}I_{R}[u]=F_{R}[u],
\end{equation}
where
\begin{align*}
F_{R}[u]&:=\int_{\R^4}\big\{(- \Delta \Delta w_{R})|u|^{2}-
\Delta[w_{R}]|u|^{4}+4\Re \overline{u_{j}} u_{k} \partial_{jk}[w_{R}]\big\}\, dx\\
&\quad +\tfrac{2}{3}\int_{\R^4}\Delta[w_{R}]|u|^{3}\, dx\\
&=:F^{c}_{R}[u]+\tfrac{2}{3}\int_{\R^4}\Delta[w_{R}]|u|^{3}\, dx.
\end{align*}
\end{lemma}
		
When $R=\infty$, we denote $F^{c}_{\infty}[u]=8G[u]$,
where
\[
G[f]:=\|\nabla f\|^{2}_{L_{x}^{2}}-\|f\|^{4}_{L_{x}^{4}}.
\]
		
\begin{lemma}\label{Virialzero}
Fix $R\in [1, \infty]$, $\theta\in \R$, $\lambda >0$ and $x_0\in\R^4$. Then we have
\[
F^{c}_{R}[e^{i\theta}\lambda  W(\lambda\cdot+x_0)]=0.
\]
\end{lemma}

\begin{proof} If $R=\infty$, then by a change of variable sand \eqref{PoQ} we get
\[
F^c_{\infty}[e^{i\theta}\lambda  W(\lambda\cdot+x_0)]=8G[e^{i\theta}\lambda  W(\lambda\cdot+x_0)]=8G[W]=0.
\]
Now assume $R\in [1, \infty)$. Fix $\theta\in \R$ and $\lambda >0$. Since 
\[
u(t,x)=e^{i\theta}\lambda  W(\lambda x+x_0)
\]
is a solution to \eqref{ECgNLS} and $I_{R}[e^{i\theta}\lambda W(\lambda\cdot+x_0)]=0$, Lemma~\ref{VirialIden} implies
\[
F^{c}_{R}[e^{i\theta}\lambda  W(\lambda\cdot+x_0)]=0,
\]
as desired. \end{proof}

\subsection{Well-posedness, stability, and concentration-compactness} We recall the following result from \cite[Theorem 4.1]{AkaIbraKikuNawa2013}. 
		
\begin{proposition}[Local well-posedness]\label{LWP}
For any $u_{0}\in H^{1}(\R^4)$, there exists a unique solution $u\in C(I, H^{1}(\R^4))$ to \eqref{NLS} on some interval $I=(-T_{{min}}, T_{{max}})\ni 0$ such that:
\begin{enumerate}[label=\rm{(\roman*)}]
\item  The solution satisfies the conservation of energy and mass 
\begin{equation*}
E(u(t))=E(u_{0})\quad  \text{and} \quad M(u(t))=M(u_{0}) \qtq{for all}t\in I.
\end{equation*}
\item $u\in L_{t}^{q}H_{x}^{1,r}(K\times \R^4)$  for every  compact time interval $K\subset I$ and for any admissible pair $(q,r)$.
\item  If $T_{{max}}$ is finite, then 
\[
\| u\|_{L_{t,x}^{6}((0,T_{{max}})\times \R^4)\cap L_{t,x}^{3}((0,T_{{max}}) \times \R^4)}=\infty.
\]
An analogous statement holds when $T_{{min}}$ is finite.
\end{enumerate}
\end{proposition}

Given an interval $I$, we define $S(I)$ as
\[
S(I):=\| u\|_{L_{t}^{\infty}L_x^2(I\times \R^4)\cap L_{t}^{2}L_x^{4}(I \times \R^4)}.
\]   
Combining Strichartz estimates, conservation of mass, and a standard continuity argument, we further obtain the following (see also \cite[Lemma~2.3]{Luo2024}). 
		
\begin{lemma}[Persistence of regularity]\label{Rpr}
Let $u: \R\times \R^4\rightarrow \C$ be a solution to \eqref{NLS} with $\|u\|_{L_{t,x}^{6}(\R\times\R^4)\cap L_{t,x}^{3}(\R\times\R^4)}\leq L<\infty$, then for any $t_0\in\R$, 
\begin{equation}
\left\||\nabla|^s u\right\|_{S\left(\mathbb{R} \right)} \leq C\left(L, M\left(u\left(t_0\right)\right)\right)\left\||\nabla|^s u\left(t_0\right)\right\|_{L_x^2}, \quad s \in [0,1].
\end{equation}
\end{lemma}

As an immediate consequence of Lemma~\ref{Rpr}, we have the following sufficient condition for scattering in $H^1$:
		
\begin{corollary}\label{ScatterCondi}
Let $u_{0}\in H^{1}(\R^4)$, and let $u$ be the solution to \eqref{NLS} with $u|_{t=0}=u_{0}$. If $u$ is global and $u\in L_{t,x}^{6}(\R\times\R^4)\cap L_{t,x}^{3}(\R\times\R^4)$, then $u$ scatters in $H^{1}(\R^4)$ as $t\to \pm \infty$.
\end{corollary}
		
We next record a stability result for \eqref{NLS} (see e.g. \cite[Lemma 2.4]{Luo2024}):
		
\begin{lemma}[Stability result, \cite{Luo2024}]\label{stabi}
Let $t_0\in I\subset\R$. Suppose $\tilde{u}:I\times\R^4\to\C$ solves
\[ 
(i\partial_{t}+\Delta) \tilde{u} =|\tilde{u}|\tilde{u}-|\tilde{u}|^{2}\tilde{u}+e,\quad  \tilde{u}(t_{0})=\tilde{u}_{0}
\]
for some $e:I\times\R^4\rightarrow \C$. Assume also that
\[
\| \tilde{u}\|_{L_t^{\infty}H_x^1(I\times\R^4)\cap L_{t,x}^{6}(I\times \R^4)\cap L_{t,x}^{3}(I \times \R^4)}\leq A.
\]
Let $ u \in C(I, H^1(\mathbb{R}^4)) $ be a solution of \eqref{NLS} defined on some interval $I \ni t_0$, and suppose 
\[
\|{u}  \|_{L_{t}^{\infty}H^{1}_{x}(I)}\leq M
\]
for some $M>0$.  Then there exists $\epsilon_{0}=\epsilon_0(A,M)>0$ such that if $0<\epsilon<\epsilon_{0}$ and
\begin{equation}\label{stability-smallness}
\|u(t_{0})-\tilde{u}(t_{0}) \|_{{H}^{1}_{x}}\leq \epsilon\qtq{and} \|\langle \nabla\rangle e  \|_{L_{t,x}^{\frac{3}{2}}(I\times\R^4)}\leq \epsilon,
\end{equation}
then 
\begin{equation}\label{SCN55}
\|\langle \nabla\rangle(u-\tilde{u})\|_{S(I)}\lesssim_{A,M}\epsilon^{\kappa}
\end{equation}
for some $\kappa\in (0,1)$.
\end{lemma}
		
\begin{remark}\label{R:stabi} The smallness condition on the initial data in \eqref{stability-smallness} may be replaced with the weaker assumption
\[
\|u(t_0)-\tilde u(t_0)\|_{H_x^1} \lesssim 1 \qtq{and} \| \langle\nabla\rangle e^{i(t-t_0)\Delta}[u(t_0)-\tilde u(t_0)]\|_{L_{t,x}^3(I\times\R^4)} \leq \eps.
\]
See, for example, Proposition~3.2 in \cite{ChengMiaoZhao2016}.
\end{remark}

We now state the following double-track profile decomposition. For a proof, see Lemma 3.6 in \cite{Luo2024}.
	
\begin{theorem}[Linear profile decomposition, \cite{Luo2024}]\label{Profi}
Let $\left\{f_{n}\right\}_{n\in \mathbb{N}}$ be a bounded sequence of  functions in $H^{1}(\R^4)$. Passing to a subsequence, there exist  $J^{\ast}\in \left\{0,1,2,\ldots\right\}\cup\left\{\infty\right\}$, nonzero profiles $\{\psi^j\}_j \subset \dot{H}^1(\mathbb{R}^4) \cup L^2(\mathbb{R}^4)$, and parameters $\{(t_n^j, x_n^j, \xi_n^j, \lambda_n^j)\}_{j,n} \subset \mathbb{R} \times \mathbb{R}^4 \times \mathbb{R}^4 \times (0, \infty)$ satisfying the following properties.

\bigskip

For any finite $1 \leq j \leq J^*$, the parameters satisfy:
\begin{itemize}
\item  $\lim\limits_{n\to\infty} |\xi_n^j| \lesssim_j 1$,
\item $\lim\limits_{n\to\infty} t_n^j =: t_\infty^j \in \{0, \pm\infty\}$,
\item $\lim\limits_{n\to\infty} \lambda_n^j =: \lambda_\infty^j \in \{0, 1, \infty\}$,
\item $t_n^j \equiv 0$ if $t_\infty^j = 0$,
\item $\lambda_n^j \equiv 1$ if $\lambda_\infty^j = 1$,
\item $\xi_n^j \equiv 0$ if $\lambda_\infty^j \in \{0, 1\}$.
\end{itemize}

Moreover, 
\[
\psi^j\in\begin{cases} \dot H^1 & \lambda_\infty^j = 0, \\ H^1 & \lambda_\infty^j = 1, \\ L^2 & \lambda_\infty^j = \infty.\end{cases}
\]

For any finite $1\leq j\leq J^*$ we have the decomposition
\begin{equation}\label{Dcom}
f_{n}=\sum^{J}_{j=1}\phi_{n}^{j}+R_n^J
\end{equation}
for each finite $1\leq J\leq J^{\ast}$, with
\begin{equation}\label{fucti}
\phi_{n}^{j} = \begin{cases}
(\lambda_n^j)^{-1} \bigl[ e^{i t_n^j \Delta} \bigl( P_{>(\lambda_n^j)^\theta} \psi^j \bigr) \bigr] \bigl( \tfrac{x - x_n^j}{\lambda_n^j} \bigr), & \lambda_\infty^j = 0, \\
\bigl[ e^{i t_n^j \Delta} \psi^j \bigr] (x - x_n^j), & \lambda_\infty^j = 1, \\
(\lambda_n^j)^{-2} e^{i x \cdot \xi_n^j} \bigl[ e^{i t_n^j \Delta} \bigl( P_{\leq (\lambda_n^j)^\theta} \psi^j \bigr) \bigr] \bigl( \tfrac{x - x_n^j}{\lambda_n^j} \bigr) & \lambda_\infty^j = \infty
\end{cases}
\end{equation}
for some $0<\theta<1$. 

Furthermore,  we have the following properties:
\begin{itemize}
\item Vanishing of the remainder:
\begin{equation}\label{Sr} 
\limsup_{J\to J^*}\limsup_{n\rightarrow\infty}\|e^{it\Delta}R_n^J\|_{L_{t,x}^{6}(\R\times\R^4)\cap L_{t,x}^{3}(\R\times\R^4)}=0.
\end{equation}
\item Orthogonality: for all $1\leq j\neq k\leq J^{\ast}$
\begin{equation}\label{Pow}
\tfrac{\lambda_n^k}{\lambda_n^j} + \tfrac{\lambda_n^j}{\lambda_n^k} + \lambda_n^k|\xi_n^j - \xi_n^k| + \bigl|t_n^k - \bigl(\tfrac{\lambda_n^k}{\lambda_n^j}\bigr)^2 t_n^j\bigr| \\
+ \bigl|\tfrac{x_n^j - x_n^k - 2t_n^k (\lambda_n^k)^2 (\xi_n^j - \xi_n^k)}{\lambda_n^k}\bigr| \to \infty.
\end{equation}		
\item Decoupling: for any $J\in \mathbb{N}$ and $s\in\{0,1\}$,
\begin{align}\label{PE11}
\| |\nabla|^s \phi_n\|_2^2 &= \sum_{j=1}^J \| |\nabla|^s \phi_{n}^{j}\|_2^2 + \| |\nabla|^s R_n^J\|_2^2 + o_n(1) \qtq{as}n\to\infty, \\
E(\phi_n) &= \sum_{j=1}^J E(\phi_{n}^{j}) + E(R_n^J) + o_n(1)\qtq{as}n\to\infty.
\end{align}
\end{itemize}

\end{theorem}
		
\subsection{Approximation by the cubic NLS} 
An important step in the proof of Theorem~\ref{Th2} is the approximation of the \eqref{NLS} by the  energy-critical NLS in the  high-scale limit.
		
We have the following scattering result for the 4$d$ energy critical NLS:
		
\begin{theorem}[Sub-threshold scattering for \eqref{ECgNLS} \cite{D5}]\label{WGPQN}
Let $u_{0}\in \dot{H}^{1}(\R^4)$. Assume 
\[
E^{c}(u_{0})<E^{c}(W)\qtq{and}\|\nabla u_{0}\|_{L^{2}}<\|\nabla W\|_{L^{2}}.
\]
Then there exists a unique global solution $u\in C(\R, \dot{H}^{1}(\R^4))$ of  the energy-critical NLS \eqref{ECgNLS} with $u|_{t=0}=u_{0}$, which satisfies
\[
\| u  \|_{L^{6}_{t,x}(\R\times\R^4)} \lesssim 1.
\]
In addition, the solution scatters in both time directions in $\dot{H}^{1}(\R^4)$.
\end{theorem}

Building on this result, the following approximation result was established in \cite{Luo2024}.
For $(\xi_0, x_0, \lambda_0) \in \mathbb{R}^4 \times \mathbb{R}^4 \times (0,\infty)$ we define the operator $g_{\xi_0,x_0,\lambda_0}$ by

\[
g_{\xi_0,x_0,\lambda_0}f(x) = \lambda_0^{-2} e^{i\xi_0 \cdot x} f(\lambda_0^{-1}(x - x_0)).
\]
Note that this operator includes the $L^2$-preserving rescaling. 
		
\begin{lemma}[Embedding nonlinear profiles I, {\cite[Lemma 4.10]{Luo2024}}] \label{P:embedding} Fix $\theta \in (0, 1)$. Let $\{t_n\}_{n\in\mathbb{N}}$ satisfy $t_n \equiv 0$ or $t_n \to \pm\infty$, and let $\{\lambda_n\}_{n\in\mathbb{N}} \subset (0,\infty)$ satisfy $\lambda_n \to 0$ as $n \to +\infty$. Suppose $\phi \in \dot{H}^1(\mathbb{R}^4)$ and obeys
\begin{align}
E^{c}(\phi) < E^{c}(W) \quad \text{and} \quad \|\nabla \phi\|_{L^2} < \|\nabla W\|_{L^2}.  \label{cdonR} 
\end{align}
Then for $n$ sufficiently large, there exists a global solution $v_n$ to \eqref{NLS} with
\begin{equation}\label{CC0}
v_n(0) = \varphi_n :=  \lambda_n g_{0, x_n, \lambda_n} e^{i t_n \Delta}P_{>\lambda^{\theta}_{n}} \phi
\end{equation}
that satisfies
\begin{align}
\|\langle \nabla \rangle v_n\|_{S(\mathbb{R}\times\mathbb{R}^4)} & 	\leq C(E^{c}(\phi)), \label{CC12} \\
\| v_n \|_{L_{t,x}^{3}(\mathbb{R}\times\mathbb{R}^4)} &\lesssim \lambda_n^{1-\theta} \label{CC1}
\end{align}
uniformly in $n$.  In addition, for every $\epsilon > 0$ there exists $N_\epsilon \in \mathbb{N}$, ${\phi}_\epsilon \in C_c^\infty(\mathbb{R} \times \mathbb{R}^4)$ and ${\psi}_\epsilon \in C_c^\infty(\mathbb{R} \times \mathbb{R}^4)$ such that for all $n \geq N_\epsilon$,
\begin{align}
\bigl\| v_n - \lambda_n^{-1} {\phi}_\epsilon \bigl( \tfrac{t}{\lambda_n^2} + t_n, \tfrac{x - x_n}{\lambda_n} \bigr) \bigr\|_{ L_{t,x}^{6}(\R\times\R^4)} &\leq \epsilon,  \\
\bigl\| \nabla v_n - \lambda_n^{-2} {\psi}_\epsilon \bigl( \tfrac{t}{\lambda_n^2} + t_n, \tfrac{x - x_n}{\lambda_n} \bigr) \bigr\|_{L_{t,x}^{3}(\R\times\R^4)} &\leq \epsilon.
\end{align}

\end{lemma}

The proof follows along the same lines as in \cite[Lemma 2.14]{AJZ} (see also \cite[Proposition 8.3]{KillipOhPoVi2017}). 

\subsection{Approximation by the quadratic NLS}
In the small-scale limit, we instead approximate \eqref{NLS} by the mass-critical NLS. 

We have the following scattering result for the 4$d$ mass-critical NLS appearing in \cite{Dod4}.

\begin{theorem}
Let $u_0 \in L^2(\mathbb{R}^4)$. Then there exists a unique global solution $u \in C(\mathbb{R}, L^2(\mathbb{R}^4))$ of the mass-critical NLS 
\begin{align}\label{MassCritical}
i\partial_t u + \Delta u = |u|u,
\end{align}
with $u|_{t=0} = u_0$, which satisfies
\[
\|u\|_{L^{3}_{t,x}(\mathbb{R} \times \mathbb{R}^4)} \lesssim_{M(u)} 1.
\]
In addition, the solution scatters in both time directions in $L^2(\mathbb{R}^4)$.
\end{theorem}

Using this result and a perturbative argument, the following small-scale approximation result was established in \cite{Luo2024}.

\begin{lemma}[Embedding nonlinear profiles II, {\cite[Lemma 4.9]{Luo2024}}]\label{nonlinearPII} Fix $\theta\in (0, 1)$. Let $\{t_n\}_{n\in\mathbb{N}}$ satisfy $t_n \equiv 0$ or $t_n \to \pm\infty$, and let $\{\lambda_n\}_{n\in\mathbb{N}} \subset (0,\infty)$ satisfy $\lambda_n \to +\infty$ as $n \to +\infty$. Let $\{\xi_n\}_{n\in\mathbb{N}} \subset \mathbb{R}^4$ be bounded. Suppose $\phi \in L^2(\mathbb{R}^4)$. 
Then for $n$ sufficiently large, there exists a global solution $u_n$ to \eqref{NLS} with
\[
u_n(0, x) = \varphi_n(x) :=  g_{\xi_n,x_n,\lambda_n} e^{i t_n \Delta}P_{\leq \lambda^{\theta}_{n}} \phi
\]
that satisfies
\begin{align}
&\|\langle \nabla\rangle u_n\|_{S(\mathbb{R})}  	\leq C(M(\phi)), \label{boundSC} \\
&\| u_n \|_{L_{t,x}^{6}(\mathbb{R}\times\mathbb{R}^4)} \lesssim \lambda_n^{-(1-\theta)} \label{Liv1}
\end{align}
uniformly in $n$. In addition, for every $\epsilon > 0$ there exists $N_\epsilon \in \mathbb{N}$ and ${\phi}_\epsilon \in C_c^\infty(\mathbb{R} \times \mathbb{R}^4)$ such that for all $n \geq N_\epsilon$,

\begin{align}
\bigl\| u_n - \lambda_n^{-2} e^{-i t |\xi_n|^2} e^{i\xi_n\cdot x}{\phi}_\epsilon \bigl( \tfrac{t}{\lambda_n^2} + t_n, \tfrac{x - x_n - 2t\xi_n}{\lambda_n} \bigr) \bigr\|_{L_{t,x}^{3}(\mathbb{R}\times\mathbb{R}^4)} &\leq \epsilon,  \\
\bigl\| \nabla u_n - i\xi_n \lambda_n^{-2} e^{-i t |\xi_n|^2} e^{i\xi_n\cdot x}{\phi}_\epsilon \bigl( \tfrac{t}{\lambda_n^2} + t_n, \tfrac{x - x_n - 2t\xi_n}{\lambda_n} \bigr) \bigr\|_{L_{t,x}^{3}(\mathbb{R}\times\mathbb{R}^4)} &\leq \epsilon.
\end{align}
\end{lemma}


\section{Blowup solutions are compact}\label{S:Compactness}

In this section we prove Proposition~\ref{P2}. 

\begin{proof}[Proof of Proposition~\ref{P2}] Suppose $u\in\mathcal{B}$, so that $u$ satisfies \eqref{Thres} and satisfies
\[
\|u\|_{L_{t,x}^6\cap L_{t,x}^3([0,\sup I_{\max})\times\R^4)} = \infty. 
\]
By Lemma~\ref{GlobalS}, we have that
\begin{align}\label{boundGC}
 &\|\nabla u(t)\|_{L^{2}}<\|\nabla W\|_{L^{2}}\quad \text{for all } t\in [0, \sup I_{\max}).
\end{align}
Using the standard arguments, we see that to prove Proposition~\ref{P2} it suffices to show that for any sequence $\left\{\tau_{n}\right\}_{n\in \mathbb{N}}\subset [0, \sup I_{\max} )$, there exist $\tilde{\lambda}_{n}\geq 1$ and $\tilde{x}_{n}\in\R^4$ such that
\begin{align}\label{LamC}
\tilde{\lambda}_{n}^{-1}u\bigl(\tau_{n}, \tfrac{x-\tilde{x}_{n}}{\tilde{\lambda}_{n}}\bigr)\qtq{converges strongly in}\dot H^1 \qtq{up to a subsequence.}
\end{align}
By $\dot H^1$ continuity of $u$, it suffices to consider $\tau_{n}\to \sup I_{\max}$. 

In light of \eqref{boundGC}, we can apply the profile decomposition (Theorem~\ref{Profi}) to the sequence $\left\{u(\tau_{n})\right\}_{n\in \mathbb{N}}$ and write
\begin{equation}\label{Dpe}
u_{n}:=u(\tau_{n})=\sum^{J}_{j=1}\phi^{j}_{n}+R_n^J\quad \text{for } J\leq J^{\ast},
\end{equation}
with all the properties stated in Theorem~\ref{Profi}.

We will show that $J^{\ast}=1$ by showing that all other possibilities lead to a contradiction. We follow closely the argument developed in \cite[Lemma 4.11]{Luo2024}.

\bigskip

\noindent\textbf{Step 1.} First, we show that $J^{\ast}\geq 1$. Indeed, suppose that $J^{\ast}=0$. Then by \eqref{Sr}, we have
\[
\| e^{it\Delta}u(\tau_{n})  \|_{L^{6}_{t,x}(\left\{t>0\right\}\times \R^{4})\cap L^{3}_{t,x}(\left\{t>0\right\}\times \R^{4})}
\to 0 \quad \text{as } n\to \infty.
\]
Using the stability result Lemma~\ref{stabi} (with $\widetilde{u}=e^{it\Delta}u(\tau_{n}) $), we obtain that for $n$ large,
\[
\|u(t+\tau_{n})\|_{L^{6}_{t,x}(\left\{t>0\right\}\times \R^{4})\cap L^{3}_{t,x}(\left\{t>0\right\}\times \R^{4})}
=\|u  \|_{L^{6}_{t,x}(\left\{t>\tau_{n}\right\}\times \R^{4})\cap L^{3}_{t,x}(\left\{t>\tau_{n}\right\}\times \R^{4})}\lesssim 1,
\]
contradicting that $u\in\mathcal{B}$.

\bigskip

\noindent\textbf{Step 2.} Suppose that $J^{\ast}\geq 2$. By the Linear Profile Decomposition Theorem~\ref{Profi} and \eqref{boundGC}, we have
\begin{align}\label{MassC}
&\lim_{n \to \infty} \biggl[\sum_{j=1}^{J} M(\phi_n^j) + M(R_n^J)\biggr] = \lim_{n \to \infty} M(u_{n}) =M(u_0),\\
\label{DECE}
&\lim_{n \to \infty} \biggl[\sum_{j=1}^{J} E(\phi_n^j) + E(R_n^J)\biggr] = \lim_{n \to \infty} E(u_{n})=E^{c}(W),\\
\label{DEG}
&\limsup_{n \to \infty} \biggl[\sum_{j=1}^{J} \| \phi_n^j\|_{\dot{H}^{1}}^2 + \| R_n^J\|_{\dot{H}^{1}}^2\biggr]= \limsup_{n \to \infty} \|u(\tau_n)\|_{\dot{H}^{1}}^2 \leq \|W\|_{\dot{H}^{1}}^2
\end{align}
for any $0\leq J \leq J^{\ast}$. By \eqref{IneW} and \eqref{DEG}, we deduce that for $n$ large,
\begin{align}\label{BounWd}
 \| \phi_n^j\|_{\dot{H}^{1}}^2\leq \frac{\|W\|_{\dot{H}^{1}}^2}{E^{c}(W)} E(\phi_n^j)
\quad \text{and} \quad
 \| R_n^J\|_{\dot{H}^{1}}^2\leq \frac{\|W\|_{\dot{H}^{1}}^2}{E^{c}(W)} E(R_n^J).
\end{align}
In particular, $\liminf\limits_{n\to \infty}E(\phi_n^j)>0$. Therefore, there exists $\epsilon^*>0$ such that
\begin{align}\label{MEC}
	E(\phi_n^j)&\leq E^{c}(W)-\epsilon^*,
\end{align}
for sufficiently large $n$ and $1\leq j\leq J$, so that each $\phi_n^j$ satisfies condition \eqref{Thres00} in Theorem~\ref{Th1}.

We use the $\phi_n^j$ to build approximate solutions to \eqref{NLS} in the following four cases:
\begin{itemize}
\item $\lambda^{j}_{n}\equiv 1$ and $t^{j}_{n}\equiv 0$;
\item $\lambda^{j}_{n}\equiv 1$ and $t^{j}_{n}\to \pm \infty$;
\item $\lambda^{j}_{n}\to \infty$;
\item and $\lambda^{j}_{n}\to 0$.
\end{itemize}

In the case where $j$ is such that $\lambda^{j}_{n}\equiv 1$ and $t^{j}_{n}\equiv 0$, \eqref{MEC} implies that we may apply Theorem~\ref{Th1}(i) to construct a global solution $v^{j}$ with initial data $v^{j}(0)=\phi^{j}$ obeying global spacetime bounds. If instead $j$ is such that $\lambda^{j}_{n}\equiv 1$ and $t^{j}_{n}\to \pm \infty$, we take the solution $v^{j}$ to \eqref{NLS} which scatters to $e^{it\Delta}\phi^{j}$ in $H^{1}$ as $t\to \pm \infty$. By \eqref{MEC} and Theorem~\ref{Th1}(i), the solution $v^{j}$ is global and satisfies uniform spacetime bounds. In either case, we define
\[
v^{j}_{n}(t,x)=v^{j}(t+t^{j}_{n}, x),\qtq{which satisfy} \| v^{j}_{n}  \|_{L^{6}_{t,x}(\R\times \R^{4})\cap L^{3}_{t,x}(\R\times \R^{4})}\lesssim_{\epsilon^*}1.
\]

Note that \eqref{BounWd} and persistence of regularity (Lemma~\ref{Rpr}) imply that
\begin{equation}\label{Estrii}
\| v^{j}_{n}  \|_{L^{6}_{t,x}(\R\times \R^{4})\cap L^{3}_{t,x}(\R\times \R^{4})}+\|\nabla  v^{j}_{n}\|_{S(\R)}\lesssim_{\epsilon^*} E(v^{j}_{n})^{\frac{1}{2}}
\end{equation}
and
\begin{equation}\label{masses}
\|  v^{j}_{n}\|_{S(\R)}\lesssim_{\epsilon^*} M(v^{j}_{n})^{\frac{1}{2}}.
\end{equation}

For the case $\lambda^{j}_{n}\rightarrow \infty$ as $n\rightarrow\infty$, we define $v_{n}^{j}$ to be the global solution to equation \eqref{NLS} with data $v_{n}^{j}(0)=\phi_{n}^{j}$ given by Lemma~\ref{nonlinearPII}, which satisfies \eqref{Liv1} and
\begin{align}\label{Lamdace}
\|  v^{j}_{n}\|_{L^{6}_{t,x}(\R\times \R^{4})\cap L^{3}_{t,x}(\R\times \R^{4})}+	\|\left\langle \nabla\right\rangle  v^{j}_{n}\|_{S(\R)}
\lesssim_{\epsilon^*} M(v^{j}_{n}).
\end{align}

Finally, for the case $\lambda^{j}_{n}\rightarrow 0$ as $n\rightarrow\infty$, we define $v_{n}^{j}$ to be the global solution to \eqref{NLS} with data $v_{n}^{j}(0)=\phi_{n}^{j}$ given by Lemma~\ref{P:embedding}. This solution satisfies \eqref{CC1} and
\begin{align}\label{Lamdace111}
\|  v^{j}_{n}\|_{L^{6}_{t,x}(\R\times \R^{4})\cap L^{3}_{t,x}(\R\times \R^{4})}+	\|\left\langle \nabla\right\rangle  v^{j}_{n}\|_{S(\R)}
\lesssim_{\epsilon^*}E^{c}(v^{j}_{n})^{\frac{1}{2}}.
\end{align}

In any case, we have
\begin{equation}\label{Apro11}
\|v_{n}^{j}(0)- \phi_{n}^{j}\|_{H^{1}_{x}}\rightarrow0 \quad \text{as } n\rightarrow\infty.
\end{equation}

We now define approximate solutions to \eqref{NLS} by
\[
u^{J}_{n}(t,x)=\sum^{J}_{j=1}v^{j}_{n}(t,x)+e^{it\Delta}R_n^J.
\]
We will now use stability (Lemma~\ref{stabi}) to contradict the fact that $u$ blows up. 

We first notice that by \eqref{Dpe} and \eqref{Apro11},
\begin{equation}\label{Aps}
\|u^{J}_{n}(0)-u_{n}\|_{H^{1}_{x}}\rightarrow0,\quad \text{as } n\rightarrow\infty
\end{equation}
for any finite $J\leq J^*$.

Moreover, using \eqref{BounWd}, \eqref{Estrii}, \eqref{masses}, \eqref{Lamdace}, and  \eqref{Lamdace111}, and arguing as in Step 1 and Step 2 of the proof of \cite[Lemma 4.11]{Luo2024}, we find:

\begin{enumerate}[label=\rm{(\roman*)}]
\item We have the following global spacetime bound:
\begin{equation}\label{Gstb}
\sup_{J}\limsup_{n\rightarrow\infty}\| {u}^{J}_{n}  \|_{L^{6}_{t,x}(\R\times \R^{4})\cap L^{3}_{t,x}(\R\times \R^{4})} 
\lesssim 1.
\end{equation}

\item For any $\eps>0$, the following holds for $J$ sufficiently large: 
\begin{equation}\label{Eimpo}
\limsup_{n\rightarrow\infty}\|\langle \nabla \rangle e^{J}_{n}  \|_{L^{\frac{3}{2}}_{x}(\R\times\R^4)}\leq\epsilon,
\end{equation}
where
\[
e^{J}_{n}:=(i\partial_{t}+\Delta) {u}^{J}_{n}+|{u}^{J}_{n}|^{2}{u}^{J}_{n}-|{u}^{J}_{n}|{u}^{J}_{n}.
\]
\end{enumerate}

Combining \eqref{Aps}, \eqref{Gstb}, and \eqref{Eimpo}, Lemma~\ref{stabi} now implies that $u\in L^{6}_{t,x}\cap L^{3}_{t,x}$, contradicting that $u$  blows up. 

\bigskip

\noindent\textbf{Step 3.} From Steps 1 and 2, we have $J^*=1$. Thus, \eqref{Dpe} simplifies to
\begin{align}\label{1Dec}
	u_{n}=u(\tau_{n})=\phi_{n}+R^{1}_{n}.
\end{align}
Now observe that
\begin{align}\label{ZerI}
\|\nabla R^{1}_{n}\|^{2}_{L^{2}}\to 0 \quad \text{as } n\to \infty.
\end{align}
Indeed, otherwise \eqref{BounWd} implies that $E( R^{1}_{n})\geq c>0$, so that $\phi_{n}$ obeys the sub-threshold condition \eqref{MEC}. The stability result Lemma~\ref{stabi} then yields global spacetime bounds for $u$, a contradiction.

We next preclude the possibility that $\lambda_{n}\to \infty$. We instead suppose $\lambda_n\to\infty$ and let $v_n$ be the solution to \eqref{NLS} given in Lemma~\ref{nonlinearPII} with initial data $\phi_{n}$.  This solution satisfies \eqref{Liv1} and
\begin{align*}
\|  v_{n}\|_{L^{6}_{t,x}(\R\times \R^{4})\cap L^{3}_{t,x}(\R\times \R^{4})}+	\|\langle \nabla\rangle  v_{n}\|_{S(\R)}
\lesssim M(v_{n}).
\end{align*}
We then consider the approximate solution of \eqref{NLS} given by
\[
u^{1}_{n}(t,x)=v_{n}(t,x)+e^{it\Delta}R_n^1.
\]
Using the argument from Step 2 and the stability result Lemma~\ref{stabi}, we derive global spacetime bounds for $u$, a contradiction.

Finally we preclude the possibility that $t_{n}\to \pm\infty$ as $n\to\infty$. Without loss of generality, suppose $t_{n}\rightarrow \infty$. We then consider the approximate solution of \eqref{NLS} given by
\[
u^{1}_{n}(t,x)=e^{it\Delta}\phi_{n}+e^{it\Delta}R_n^1.
\]
Note that $u^{1}_{n}(0)=u_{n}$. If $\lambda_{n}\equiv1$, then monotone convergence implies
\[
\| e^{it\Delta}\phi_{n}  \|_{L^{6}_{t,x}([0, \infty)\times \mathbb{R}^{4})\cap L^{3}_{t,x}([0, \infty)\times \mathbb{R}^{4})}\rightarrow 0
\]
as $n\rightarrow\infty$. On the other hand, if $\lambda_{n}\to 0$, a change of variables and monotone convergence yield
\begin{equation*}
\begin{split}
&\| e^{it\Delta}\phi_{n} \|_{L^{6}_{t,x}([0, \infty)\times \mathbb{R}^{4})\cap L^{3}_{t,x}([0, \infty)\times \mathbb{R}^{4})} \\
&\quad = \|e^{it\Delta} P_{\geq (\lambda_{n})^{\theta}}\phi\|_{L^{6}_{t,x}([t_{n},\infty)\times \mathbb{R}^{4})} \\
&\quad \quad + \lambda_{n}\|e^{it\Delta} P_{\geq (\lambda_{n})^{\theta}}\phi\|_{L^{3}_{t,x}([t_{n},\infty)\times \mathbb{R}^{4})} \rightarrow 0,
\end{split}
\end{equation*}
as $n\rightarrow\infty$. In either case, from \eqref{Sr} we conclude that
\begin{align*}
\limsup_{n\rightarrow\infty}\| {u}^{1}_{n}  \|_{L^{6}_{t,x}([0, \infty)\times \mathbb{R}^{4})\cap L^{3}_{t,x}([0, \infty)\times \mathbb{R}^{4})} =0.
\end{align*}
Thus, the argument from Step 2 applies once again and yields global spacetime bounds for $u$, a contradiction.  Hence $t_n\equiv 0$.

Setting $\tilde{\lambda}_n=\frac{1}{\lambda_n}$ and $\tilde{x}_n=x_n$, we obtain \eqref{LamC} and conclude the proof. \end{proof}


\section{Global existence}\label{no ftb}

In this section we prove Proposition~\ref{P3}.

\begin{proof} Suppose $u\in\mathcal{B}$ but $T_{\max}=\sup I_{\max}<\infty$. (The case of negative finite time blowup is similar).

We consider a sequence $t_n \to T_{\text{max}}$. By the finite blow-up criterion (cf. Proposition~\ref{LWP}~(iii)), we must have $\|u\|_{S([0,T_{\text{max}}))} = \infty.$
From Proposition~\ref{P2}, there exists $\lambda(t) \geq 1$ and $x(t)\in\R^4$ such that the set defined by \eqref{CompactX} is precompact in $\dot H^{1}(\mathbb{R}^4)$. 

We now claim that
\begin{equation}\label{Claii}
\lim_{t \to T_{\text{max}}} \lambda(t) = \infty.
\end{equation}

Suppose instead that \eqref{Claii} were false. Then there exists a sequence $t_n \to T_{\max}$ such that $\lambda_n:=\lambda(t_n) \to \lambda_0 \in [1, \infty)$.  

By $\dot H^1$-compactness, we can write
\[
u(t_n,x) = \lambda_n \Phi(\lambda_n(x-x_n)) + e_n
\]
with $\Phi\in\dot H^1\backslash\{0\}$ and $e_n\to 0$ strongly in $\dot H^1$.  On the other hand, by the profile decomposition argument used to prove Proposition~\ref{P2}, we may also write
\begin{equation}\label{alt1}
u(t_n) = \phi(\cdot-y_n) + R_n
\end{equation}
where $\phi\in H^1\backslash\{0\}$ or
\begin{equation}\label{alt2}
u(t_n,x) = \beta_n^{-1}[P_{>\beta_n^\theta}\phi](\tfrac{x-y_n}{\beta_n}) +R_n
\end{equation}
where $\phi\in\dot H^1\backslash\{0\}$, $\beta_n\to 0$, and $\theta\in(0,1)$.  In either case $R_n\to 0$ strongly in $\dot H^1$ (but not necessarily in $L^2$) and $\|e^{it\Delta}R_n\|_{L_{t,x}^3 \cap L_{t,x}^6}\to 0$. 

We note, however, that \eqref{alt2} cannot occur.  Indeed, since $\lambda_n\to \lambda_0$, equating the two representations of $u(t_n)$ and changing variables would show that $\Phi$ is equal to a sequence converging weakly to zero in $\dot H^1$, contradicting that $\Phi\neq 0$. 

Thus (redefining $\phi$ to absorb the limiting scale $\lambda_0$), we may write 
\begin{align}\label{1Dec0}
u(t_n) = \phi_n + R_n\qtq{with} \|R_n\|_{\dot{H}^1}\to 0,
\end{align}
with $\phi_n=\phi(\cdot-y_n)\in H^1$ and $R_n\in H^1$ satisfying the properties detailed above. 

We now consider the approximate solution
\[
U_n(t,x) = e^{it\Delta}(\phi_n + R_n),
\]
which satisfies $U_n(0) = u(t_n)$ and obeys global spacetime estimates. We now claim that
\[
e_n:=(i\partial_t+\Delta) U_n +|U_n|^2 U_n - |U_n| U_n = |U_n|^2 U_n - |U_n| U_n
\]
satisfies
\[
\|\langle\nabla\rangle e_n \|_{L_{t,x}^{\frac32}([0,T_{\max}-t_n)\times\R^4)} \to 0 \qtq{as}n\to\infty. 
\]
Indeed, this follows from the vanishing property of $R_n$ under the free evolution and the fact that
\[
\| e^{it\Delta} \phi_n\|_{L_{t,x}^3\cap L_{t,x}^6([0,T_{\max}-t_n)\times\R^4)} \to 0,
\]
which is a consequence of Strichartz estimates and Sobolev embedding, the facts that $\phi\in H^1$ and $t_n\to T_{\max}$, and the Dominated Convergence Theorem. 

An application of the stability theory (Lemma~\ref{stabi}) therefore yields uniform spacetime bounds for $u(t_n+t)$ on the interval $[0,T_{\max}-t_n)$, or equivalently for $u$ on the interval $[0,T_{\max})$. However, this contradicts the fact that $u$ blows up at $T_{\max}$. Thus we conclude that \eqref{Claii} holds. 

We now let $R > 0$ and define the localized mass
\[
M_R(t) = \int_{\mathbb{R}^4} |u(t,x)|^2 \, \xi\bigl(\tfrac{x}{R}\bigr) \, dx, \qtq{for} t \in [0, T_{\text{max}}),
\]
where $\xi$ is a smooth cutoff function with $\xi(x) = 1$ for $|x| \leq 1$ and $\xi(x) = 0$ for $|x| \geq 2$. We note that
\[
M_R'(t) = \tfrac{2}{R} \Im \int_{\mathbb{R}^4} \bar{u} \, \nabla u \cdot (\nabla\xi)\bigl(\tfrac{x}{R}\bigr) \, dx.
\]
Using uniform $H^1$ bounds, it follows that $|M_R'(t)| \lesssim_W 1$. Consequently, for $t_1, t_2 \in [0, T_{\text{max}})$,
\begin{align}\label{MassBound}
	M_R(t_1) = M_R(t_2) + \int_{t_1}^{t_2} M_R'(s) \, ds \lesssim_W M_R(t_2) + |t_1 - t_2|.
\end{align}

Furthermore, as was shown in \cite[Case 1 of Proposition~5.3]{KM1}, the compactness of the solution and the fact that $\lambda(t) \to \infty$ imply that for any $R > 0$,
\[
\limsup_{t \to T_{\text{max}}} \int_{|x|\leq R} |u(t,x)|^2 \, dx = 0.
\]
Thus, $M_R(t_2) \to 0$ as $t_2 \to T_{\text{max}}$. Applying this to \eqref{MassBound} yields
\[
M_R(t_1) \lesssim_W |T_{\text{max}} - t_1|.
\]
Letting $R \to \infty$ and using the monotone convergence theorem, we obtain
\[
M(u_0) = \|u(t_1)\|_{L^2}^2 \lesssim_W |T_{\text{max}} - t_1|.
\]
Finally, letting $t_1 \to T_{\text{max}}$ implies $M(u_0) = 0$, yielding the desired contradiction. \end{proof}


\section{Scale vs. modulation}\label{PreClu}

In this section we prove Proposition~\ref{P4}.

\begin{proof}[Proof of Proposition~\ref{P4}] Suppose $u\in\mathcal{B}$.  We suppose towards a contradiction that $t_n\to\infty$ is such that $\lambda(t_n)\to\infty$ but
\begin{equation}\label{deltabound11}
	\delta(u(t_n)) \geq c > 0
\end{equation}
along some subsequence. To simplify notation, let us write $\lambda_n=\lambda(t_n)$.  Furthermore, as the translation parameters play no role in the argument, let us assume without loss of generality that $x(t_n)\equiv 0$ (and similarly we will ignore the translation parameters when we use the linear profile decomposition).

By $\dot H^1$-compactness, we can write
\[
u(t_n,x) = \lambda_n \Phi(\lambda_n x) + e_n,
\]
with $\Phi\in\dot H^1\backslash\{0\}$ and $e_n\to 0$ strongly in $\dot H^1$.

By the profile decomposition argument used to prove Proposition~\ref{P2}, we also have that either
\begin{equation}\label{vss1}
u(t_n) = \phi + R_n
\end{equation}
where $\phi\in H^1\backslash\{0\}$, or
\begin{equation}\label{vss2}
u(t_n,x) = \beta_n^{-1}[P_{>\beta_n^\theta}\phi](\tfrac{x}{\beta_n}) + R_n,
\end{equation}
where $\phi\in\dot H^1\backslash\{0\}$, $\beta_n\to 0$, and $\theta\in(0,1)$.  In either case we have $R_n\to 0$ strongly in $\dot H^1$ and $\|e^{it\Delta}R_n\|_{L_{t,x}^3 \cap L_{t,x}^6} \to 0$.  

We can rule out the scenario \eqref{vss1}, as in this case we would have
\[
\phi = \lambda_n\Phi(\lambda_n x) + e_n - R_n \rightharpoonup 0 \qtq{weakly in}\dot H^1,    
\]
contradicting that $\phi\neq 0$. Thus scenario \eqref{vss2} holds, and in fact we obtain
\[
\lambda_n^{-1}\beta_n^{-1}[P_{>\beta_n^\theta}\phi](\tfrac{x}{\lambda_n\beta_n}) \to \Phi(x) \qtq{strongly in}\dot H^1. 
\]
From this we can see that without loss of generality, we may take $\beta_n=\lambda_n^{-1}$. 

We now observe that $E^c(\phi) = E^c(W)$, and in light of \eqref{deltabound11} we have 
\[
\sup_n \|u(t_n)\|_{\dot H^1}^2 \leq \|W\|_{\dot H^1}^2 - c,\qtq{so that} \|\phi\|_{\dot H^1} < \|W\|_{\dot H^1}.
\]

We let $w$ be the solution of the energy-critical NLS \eqref{ECgNLS} with $w|_{t=0} = \phi$. By the classification of threshold solutions for \eqref{ECgNLS} (see \cite[Theorem~1.3]{MaMiaoMurphyZheng2025}), we have that either $w$ scatters in both time directions or $w=W^-$ modulo symmetries, where $W^-$ is the heteroclinic orbit that scatters as $t\to-\infty$ and converges exponentially to $W$ as $t\to\infty$. Thus $w$ is global and satisfies
\[
\|w\|_{L^{6}_{t,x}([0,\infty)\times\mathbb{R}^4)} < \infty \quad \text{or} \quad \|w\|_{L^{6}_{t,x}((-\infty,0]\times\mathbb{R}^4)} < \infty
\]
(or both). Without loss of generality, we assume the forward-in-time bound holds.

We next let $w_n$ be the solution to \eqref{ECgNLS} with initial data 
\[
w_n(0) = P_{>\beta_n^\theta}\phi
\]
for some $\theta\in(0,1)$. As $\beta_n\to0$, we have that $w_n(0)\to \phi$ strongly in $\dot H^1$.  Thus by the stability theory for \eqref{ECgNLS} (cf. \cite[Theorem 2.14]{KM1}) we can derive that that
\[
\|\nabla w_n\|_{S([0,\infty))} \lesssim 1 \qtq{for all large}n.
\]
Furthermore, since
\[
\|P_{>(\beta_n)^\theta}\phi\|_{L^2} \lesssim \beta_n^{-\theta} \|\nabla \phi\|_{L^2},
\]
persistence of regularity for \eqref{ECgNLS} implies
\begin{align}\label{Prg1}
\|w_n\|_{S([0,\infty))} \leq C(\|\phi\|_{\dot{H}^1_x}) \beta_n^{-\theta}.
\end{align}

Now consider 
\[
\tilde{u}_n(t, x) = \beta_n^{-1} w_n(\beta_n^{-2} t,\beta_n^{-1} x).
\]
Then $\tilde{u}_n$ is the solution to \eqref{ECgNLS} with initial data 
\[
\tilde{u}_n(0,x) = \beta_n^{-1} [P_{>\beta_n^\theta}\phi](\beta_n^{-1} x).
\]

We now observe that 
\[
\|\nabla\tilde{u}_n\|_{S([0,\infty))} = \|\nabla w_n\|_{S([0,\infty))} \lesssim 1
\]
and (from \eqref{Prg1}) 
\[
\|\tilde{u}_n\|_{S([0,\infty))} = \beta_n\|w_n\|_{S([0,\infty))} \lesssim \beta_n^{1-\theta}\to 0 \qtq{as}n\to\infty.
\]
Finally, as 
\[
\|\langle \nabla\rangle |\tilde u_n|\tilde u_n \|_{L^{\frac{3}{2}}_{t,x}([0,\infty)\times\mathbb{R}^4)}  \lesssim \|\langle \nabla\rangle\tilde{u}_n\|_{L^{3}_{t,x}}\|\tilde{u}_n\|_{L^{3}_{t,x}} \lesssim \beta_n^{1-\theta}\to 0
\]
as $n\to\infty$, we may apply stability (Lemma~\ref{stabi}) to deduce the existence of the solutions $v_n$ to \eqref{NLS} on $[0,\infty)$ with 
\[
v_n(0)=\beta_n^{-1} [P_{>\beta_n^\theta}\phi](\beta_n^{-1} x)
\]
such that for large $n$, 
\[
\|\langle \nabla \rangle v_n\|_{S([0,\infty))} \lesssim 1 	\qtq{and} \lim_{n \to \infty} \| v_n \|_{L_{t,x}^{3}([0,\infty)\times\mathbb{R}^4)} = 0.
\]

Now let
\[
U_n(t,x) = v_n(t,x) + e^{it\Delta} R_n.
\]
These functions obey global spacetime bounds uniformly in $n$. Moreover, using the scattering of the solutions $v_n$ and the vanishing property of $R_n$, we can prove that the $U_n$ asymptotically solve \eqref{NLS} as $n\to\infty$ (as in Step 2 of the proof of Proposition~\ref{P2}). Finally, by Strichartz estimates and the vanishing property of the $R_n$, 
\begin{align*}
\|\langle\nabla\rangle e^{it\Delta}[U_n(0)-u(t_n)]\|_{L_{t,x}^3} & \lesssim \| e^{it\Delta} R_n\|_{L_{t,x}^3} + \|R_n\|_{\dot H^1}\to 0
\end{align*}
as $n\to\infty$.  Thus, by stability (see Lemma~\ref{stabi} and Remark~\ref{R:stabi}), we conclude that 
\[
\|u\|_{L_{t,x}^3\cap L_{t,x}^6([t_n,\infty)\times\R^4)} \lesssim 1
\]
uniformly for large $n$, contradicting the fact that $u$ blows up forward in time. \end{proof}


\section{Small modulation along a sequence}

In this section we prove Proposition~\ref{P5}.

\begin{proof}[Proof of Proposition~\ref{P5}] Suppose $u\in\mathcal{B}$, and suppose towards a contradiction that 
\[
\inf_{t\in[0,\infty)}\delta(u(t))\geq c_0>0.
\]
By Proposition~\ref{P4}, this implies that 
\[
\sup_{t\in[0,\infty)}\lambda(t) \lesssim 1,
\]
and hence we may assume that $\lambda(t)\equiv 1$.

We now establish an interaction Morawetz estimate for $u(t)$ that will force $u\equiv 0$ (a contradiction).

Let $\psi \in C_0^\infty(\mathbb{R})$ be a radial function such that $\psi(x) = 1$ when $|x| \leq 1$ and $\psi(x) = 0$ when $|x| > 2$. Define
\[
\phi(x - y) = \int \psi^2(x - s)\psi^2(y - s)  ds.
\]
Then $\phi(x)$ is supported on $|x| \leq 4$. For $1 \leq R \leq R_0$, define the interaction Morawetz potential:
\[
M_R(t) = \int |u(t,y)|^2 \phi\bigl( \tfrac{|x-y|}{R} \bigr)(x-y)_j \Im[\overline{u}\partial_j u](t,x)  \, dx  dy
\]
and let
\[
M(t) = \int_{1 \leq R \leq R_0} \tfrac{1}{R} M_R(t) \, dR.
\]

By H\"older's inequality and Sobolev embedding, we have
\[
\sup_{t \in \mathbb{R}} |M_R(t)| \lesssim R^4.
\]

A direct calculation yields:
\begin{align}
\tfrac{d}{dt}& M_R(t)\nonumber \\
& =  2 \int |u|^2(t,y) \phi\bigl( \tfrac{|x-y|}{R} \bigr) \left[ |\nabla u|^2 - |u|^4 + |u|^3 \right](t,x)  \, dx  dy \label{estq1} \\
& - 2 \int \Im[\overline{u}\partial_j u](t,y) \phi\bigl( \tfrac{|x-y|}{R} \bigr) \Im[\overline{u}\partial_j u](t,x)  \, dx  dy \label{estq2} \\
& + 2 \int |u|^2(t,y) \phi'\bigl( \tfrac{|x-y|}{R} \bigr) \tfrac{(x-y)_j(x-y)_k}{R|x-y|} \Re[(\partial_j \overline{u}\partial_k u)](t,x)  \, dx\, dy \label{estq3} \\
& + \int |u|^2(t,y) \phi'\bigl( \tfrac{|x-y|}{R} \bigr) \tfrac{|x-y|}{R} \left[ |\nabla u|^2 - \tfrac{1}{2}|u|^4 + \tfrac{1}{2}|u|^3 \right](t,x)  \, dx\,  dy \label{estq4} \\
& - 2 \int \Im[\overline{u}\partial_k u](t,y) \phi'\bigl( \tfrac{|x-y|}{R} \bigr) \tfrac{(x-y)_j(x-y)_k}{R|x-y|} \Im[\overline{u}\partial_j u](t,x)  \, dx\,  dy \label{estq5} \\
& - \tfrac{1}{2} \int |u|^2(t,y) \Delta \bigl[ 4\phi\bigl( \tfrac{|x-y|}{R} \bigr) + \phi'\bigl( \tfrac{|x-y|}{R} \bigr) \tfrac{|x-y|}{R} \bigr] |u|^2(t,x)  \, dx\,  dy. \label{estq6}
\end{align}

Let $I = (0,K) \subseteq \mathbb{R}$ and $1 \leq R_0 \leq K^{\frac15}$.
We first  consider terms \eqref{estq3}--\eqref{estq5}. From the definition of $\phi$, we have
\begin{align*}
\int_{1 \leq R \leq R_0} &\tfrac{1}{R} \bigl| \phi'\bigl( \tfrac{|x-y|}{R} \bigr) \tfrac{(x-y)_j(x-y)_k}{|x-y|R} \bigr|\,  dR \\
&\lesssim \int_{1 \leq R \leq R_0} \tfrac{1}{R} \bigl| \phi'\bigl( \tfrac{|x-y|}{R} \bigr) \tfrac{|x-y|}{R} \bigr|\,  dR \lesssim 1,
\end{align*}
and this is supported on $|x-y| \lesssim R_0$. Hence,
\begin{equation}
\bigl| \int_I \int_{1 \leq R \leq R_0} \tfrac{1}{R} [\eqref{estq3}+\eqref{estq4}+\eqref{estq5}]\,  dR\,  dt \bigr| \lesssim K. \label{estqq1}  
\end{equation}

Next we estimate the term \eqref{estq5}. Notice that
\begin{align*}
&\biggl| \int_{1 \leq R \leq R_0} \tfrac{1}{R} \Delta \bigl[ 4\phi\bigl( \tfrac{|x-y|}{R} \bigr) + \phi'\bigl( \tfrac{|x-y|}{R} \bigr) \tfrac{|x-y|}{R} \bigr]\,  dR \biggr| \\
&=\biggl| \int_{1 \leq R \leq R_0} \tfrac{1}{R} \bigl[ \phi''\bigl( \tfrac{|x-y|}{R} \bigr) \tfrac{|x-y|}{R^3} + \phi''\bigl( \tfrac{|x-y|}{R} \bigr) \tfrac{6}{R^2} + \phi'\bigl( \tfrac{|x-y|}{R} \bigr) \tfrac{15}{R|x-y|} \bigr]\,  dR \biggr| \\
&\lesssim \tfrac{1}{1+|x-y|^2} + \tfrac{1}{|x-y| + |x-y|^2}.
\end{align*}
This, together with the Hardy--Littlewood--Sobolev inequality, Sobolev embedding, and the fact taht $u \in L_t^\infty H_x^1$ implies that
\begin{align}
&\int_I \int_{1 \leq R \leq R_0} \eqref{estq6}\, \tfrac{1}{R}\,  dR\,  dt\notag \\
&\lesssim \int_I \int_{1 \leq R \leq R_0} |u|^2(t,x) \bigl( 1 + \tfrac{1}{|x-y|^2} \bigr) |u|^2(t,y)  \, dx \, dy  \, dt\notag \\
&\lesssim K \bigl( \|u\|_{L_t^\infty L_x^2}^4 + \|u\|_{L_t^\infty L_x^{8/3}}^4 \bigr) \lesssim K. \label{estqq2}
\end{align}

Finally we handle the terms \eqref{estq1} and \eqref{estq2}. Let us we write $\psi(x) = \psi(|x|)$ for $x \in \mathbb{R}^4$. For each $s$ and $t$, there exists $\xi(s,t)$ such that
\[
\int \psi^2\bigl( \tfrac{x}{R} - s \bigr) \Im[e^{ix\cdot \xi(s,t)} {\bar u} \nabla (e^{ix\cdot \xi(s,t)} u)](t,x)  \, dx = 0. 
\]
Moreover, for any $s$, $t$, and $\xi(s,t)$, the quantity
\[
\int \psi^2( \tfrac{x}{R} - s ) \psi^2( \tfrac{y}{R} - s ) \left[ |\nabla u(t,x)|^2 |u(t,y)|^2 - \Im[\bar u\nabla u](t,x) \Im[\bar u\nabla u](t,y) \right] \,dx\,dy
\]
is invariant under the Galilean transformation $u(t,x) \mapsto e^{ix\cdot \xi(s,t)} u(t,x)$. Therefore, for any given $s$, $t$ we can choose $\xi(s,t)$ to eliminate the momentum squared terms. Then, integration by parts gives:
\begin{align*}
&\int \psi^2\bigl( \tfrac{x}{R} - s \bigr) [|\nabla (e^{ix\cdot \xi(s,t)} u(t,x))|^2 - |u(t,x)|^4 + |u(t,x)|^3]  \, dx \\
&= \int \bigl| \nabla \bigl( \psi\bigl( \tfrac{x}{R} - s \bigr) e^{ix\cdot \xi(s,t)} u(t,x) \bigr) \bigr|^2  \, dx - \int \psi^2\bigl( \tfrac{x}{R} - s \bigr) |u(t,x)|^4  \, dx \\
& + \int \psi^2\bigl( \tfrac{x}{R} - s \bigr) |u(t,x)|^3  \, dx + \int |u(t,x)|^2 \psi\bigl( \tfrac{x}{R} - s \bigr) \Delta_x \psi\bigl( \tfrac{x}{R} - s \bigr)  \, dx\notag\\
&=:A+\int |u(t,x)|^2 \psi\bigl( \tfrac{x}{R} - s \bigr) \Delta_x \psi\bigl( \tfrac{x}{R} - s \bigr)  \, dx.
\end{align*}
Recalling that $\|u\|_{L_t^\infty \dot{H}_x^1} \leq  \|W\|_{\dot{H}^1}$. Thus by sharp Sobolev embedding,
\[
\|u\|_{L^4(\mathbb{R}^4)} \leq  \|W\|_{L^4(\mathbb{R}^4)}. 
\]
Now $v = e^{ix\cdot \xi(s,t)} u$. Then by H\"older's inequality,
\begin{align*}
A &= \|\nabla (\psi v)\|_{L^2}^2 - \|\psi^2 |u|^4\|_{L^1} + \|\psi^2 |u|^3\|_{L^1}  \geq \|\psi^2 |u|^3\|_{L^1}.
\end{align*}
For the term related to $\Delta_x \psi$, we observe that
\[
\int \psi\bigl( \tfrac{x}{R} - s \bigr) \Delta_x \psi\bigl( \tfrac{x}{R} - s \bigr) \cdot \psi^2\bigl( \tfrac{y}{R} - s \bigr)  ds \lesssim \psi\bigl( \tfrac{|x-y|}{4R} \bigr) \bigl( \tfrac{1}{R^2} + \tfrac{1}{R|x-y|} \bigr),
\]
and hence
\begin{align}
&\biggl| \int_I \int_{1 \leq R \leq R_0}  \int \int |u(t,x)|^2 \psi\bigl( \tfrac{x}{R} - s \bigr) \Delta_x \psi\bigl( \tfrac{x}{R} - s \bigr)  |u(t,y)|^2  \, dx\,  dy\,  ds\,  \tfrac{dR}{R} \, dt \biggr|\notag \\
&\lesssim \int_I \int \int_{1 \leq R \leq R_0} |u(t,x)|^2 \bigl[ \psi\bigl( \tfrac{|x-y|}{4R} \bigr) \bigl( \tfrac{1}{R^2} + \tfrac{1}{R|x-y|} \bigr) \bigr] |u(t,y)|^2 \, \tfrac{dR}{R}\,   dx\,  dy\,  dt\notag \\
&\lesssim \int_I \int |u(t,x)|^2 \bigl( 1 + \tfrac{1}{|x-y|^2} \bigr) |u(t,y)|^2  \, dx\,  dy\, dt \lesssim K. \label{estqq3}
\end{align}

On the other hand, noticing that for any $|x-y| \leq R_0^{\frac12}$ and $R \geq R_0^{\frac12}$, we have:
\[
\int \psi^2\bigl( \tfrac{x}{R} - s \bigr) \psi^2\bigl( \tfrac{y}{R} - s \bigr)  ds \gtrsim 1.
\]
Thus,
\begin{align*}
&\int_I \int_{R_0^{1/2} \leq R \leq R_0}  \int |u(t,y)|^2 \psi^2\bigl( \tfrac{x}{R} - s \bigr) \psi^2\bigl( \tfrac{y}{R} - s \bigr) |u(t,x)|^3\,  ds\,  dx\,  dy \, \tfrac{dR}{R}\,  dt \\
&\gtrsim \int_I \int_{R_0^{1/2} \leq R \leq R_0} \int |u(t,y)|^2 |u(t,x)|^3\,  dx\,  dy\,  \tfrac{dR}{R}\,  dt \\
&\gtrsim \ln R_0 \int_I \int_{|x-y| \leq R_0^{1/2}} \int |u(t,y)|^2 |u(t,x)|^3\,  dx\,  dy\,  dt \\
&\gtrsim \ln R_0 \int_I \biggl[ \int_{|x-x(t)| \leq R_0^{1/2}} |u(t,x)|^3\,  dx \cdot \int_{|y-x(t)| \leq R_0^{1/2}} |u(t,y)|^2\,  dy \biggr]\,  dt,
\end{align*}
which together with \eqref{estqq1}--\eqref{estqq3} implies:
\begin{align*}
R_0^4 &\gtrsim \int_{1 \leq R \leq R_0}  \sup_t |M_R(t)|\,  \tfrac{dR}{R} \\
&\gtrsim \biggl| \int_I \int_{1 \leq R \leq R_0} \tfrac{d}{dt} M_R(t)\,  dt  \tfrac{dR}{R} \biggr| \\
&\gtrsim \ln R_0 \int_I \biggl[ \int_{|x-x(t)| \leq R_0^{1/2}} |u(t,x)|^3\,  dx \int_{|y-x(t)| \leq R_0^{1/2}} |u(t,y)|^2\,  dy \biggr] dt - K.
\end{align*}

Since $1 \leq R_0 \leq K^{1/5}$, we have:
\[
\ln R_0 \int_I \biggl[ \int_{|x-x(t)| \leq R_0^{1/2}} |u(t,x)|^3\,  dx \cdot \int_{|y-x(t)| \leq R_0^{1/2}} |u(t,y)|^2\,  dy \biggr]\,  dt \lesssim K. 
\]
Therefore, there exists sequences $t_n \geq 0$ and $R_{0,n}\to\infty$ such that either
\begin{equation}\label{sequencefinal1}
\int_{|x-x(t_n)| \leq R_0^{1/2}}|u(t_n,x)|^3\,  dx \to 0
\end{equation}
or
\begin{equation}\label{sequencefinal2}
    \int_{|y-x(t_n)| \leq R_0^{1/2}} |u(t_n,y)|^2\,  dy \to 0.
\end{equation}
As $\{u(t_n,x-x(t_n))\}$ is precompact in $\dot{H}^1$ and bounded in $L^2$, there exists a non-zero $\phi\in H^1$ such that
\[
\|u(t_n,x-x(t_n))-\phi\|_{\dot{H}^1}\to 0
\]
up to a subsequence. This yields a contradiction, as either \eqref{sequencefinal1} or \eqref{sequencefinal2} would imply that $\phi\equiv0$.
\end{proof}


\section{Small modulation for all time}\label{S:SMAT}

In this section we prove Proposition~\ref{P6}. 

\begin{proof}[Proof of Proposition~\ref{P6}] We recall the definition of $\mathcal{B}$ and $m_c$ from Section~\ref{S:strategy}.  We suppose $\mathcal{B}\neq\emptyset$.  Our goal is then to find a solution $v\in\mathcal{B}$ satisfying $\delta(v(t))\leq\eta_*$ for all $t\geq 0$. 

\bigskip

\textbf{Part 1.} Our approach utilizes the notion of a minimal mass blowup solution in $\mathcal{B}$.  We begin by choosing $u_n\in\mathcal{B}$ such that $M(u_n)\to m_c\in[0,\infty)$. We then choose a sequence $t_n^0\to\infty$ such that 
\[
\|u_n\|_{L_{t,x}^6\cap L_{t,x}^3([0,t_n^0)\times\R^4))}\geq n. 
\]
For each $n$, we now consider two possible cases:
\begin{itemize}
\item[(i)] $\delta(u_n(t))\leq \eta_*$ for all $t\geq t_n^0$, or
\item[(ii)] there exists $t_n\geq t_n^0$ such that $\delta(u(t_n))>\eta_*$.
\end{itemize}

If case (i) occurs for some $n$, then the solution $v\in\mathcal{B}$ defined by $v(t):=u_n(t+t_n^0)$ satisfies $\delta(v(t))\leq \eta_*$ for all $t\geq 0$, yielding the desired solution. 


It therefore remains to consider the possibility that case (ii) occurs for all $n$.  In this case, by Proposition~\ref{P4} we may assume without loss of generality that $\lambda_n(t_n)\equiv 1$ (where $\lambda_n$ is the frequency scale function for $u_n$).  We define the solutions $v_n(t)=u_n(t+t_n)$, which satisfy
\begin{equation}\label{vnbu}
\|v_n\|_{L_{t,x}^6\cap L_{t,x}^3((0,\infty)\times\R^4)} =\infty\qtq{and} \|v_n\|_{L_{t,x}^6\cap L_{t,x}^3((-\infty,0)\times\R^4)}\to\infty.
\end{equation}
We apply the profile decomposition (Theorem~\ref{Profi}) to the sequence $\{v_n(0)\}$.  This yields
\[
v_n(0) = \sum_{j=1}^J \phi_n^j + R_n^J
\]
for $J\leq J^*$, with all of the properties stated in Theorem~\ref{Profi}.  The argument now follows along the lines of the proof of Proposition~\ref{P2}. 

If $J^*=0$, then the vanishing property of the remainders, together with the stability lemma (applied to the approximate solutions $e^{it\Delta}v_n(0)$), we obtain that the solutions $v_n$ scatter in both time directions, contradicting \eqref{vnbu}.  Thus we have $J^*\geq 1$.

We now claim that $J^*=1$. We suppose instead that $J^*\geq 2$. In the proof of Proposition~\ref{P2}, the profile decomposition was applied to a single $H^1$ solution satisfying the energy constraint \eqref{Thres}, while here the profile decomposition is applied to a sequence of solutions all satisfying \eqref{Thres} and with mass converging to $m_c$.  It follows that we can repeat the construction of the approximate solutions (denoted $u_n^J$ in the proof of Proposition~\ref{P2}) satisfying global spacetime bounds.  We then apply the stability theorem to deduce global spacetime bounds for the solutions $v_n$, contradicting \eqref{vnbu}. 

We therefore obtain a decomposition of the form
\[
v_n(0) = \phi_n + R_n, 
\]
with $R_n\to 0$ strongly in $\dot H^1$ and satisfying the vanishing property under the free evolution. We can again argue as in the proof of Proposition~\ref{P2} to preclude the possibility that the time translation parameters or physical scale associated to $\phi_n$ diverge.  Furthermore, the spatial translation parameters play no meaningful role in the argument and thus we assume without loss of generality that they vanish as well.  Then we have that either 
\begin{equation}\label{cmcc1}
\phi_n \equiv \psi \qtq{for some} \psi\in H^1\backslash\{0\}
\end{equation}
or
\begin{equation}\label{cmcc2}
\phi_n = \beta_n^{-1} [P_{>\beta_n^\theta}\psi](\tfrac{\cdot}{\beta_n})\qtq{for some}\beta_n\to0,\quad  \psi\in\dot H^1\backslash\{0\}.
\end{equation}
In case \eqref{cmcc2}, observe that by Bernstein estimates, we have 
\[
\|\phi_n\|_{L^2} \lesssim \beta_n^{1-\theta}\|\psi\|_{\dot H^1} \to 0. 
\]

Note that in either case we have that $E(R_n)\to 0$, from which we derive $E(\psi)=E^c(W)$ in the setting of \eqref{cmcc1} and $E(\phi_n)\to E^c(\psi)=E^c(W)$ in the setting of \eqref{cmcc2}. We further claim that $\|\psi\|_{\dot H^1}<\|W\|_{\dot H^1}$. Indeed, this follows from strong $\dot H^1$ convergence and by the construction of the sequence $t_n$, which guarantees
\[
\|v_n(0)\|_{\dot H^1}^2 = \|u_n(t_n)\|_{\dot H^1}^2\leq \|W\|_{\dot H^1}^2-\eta_*. 
\]

In the setting of of Proposition~\ref{P2}, either of \eqref{cmcc1} or \eqref{cmcc2} was acceptable to complete the proof.  Here we must additionally rule out the possibility of \eqref{cmcc2}.  Thus we suppose that \eqref{cmcc2} holds and look for a contradiction.  In fact, at this point we can argue much as in the proof of Proposition~\ref{P4} to deduce uniform spacetime bounds for the solutions $u_n$ on either $[0,t_n)$ or $[t_n,\infty)$ (a contradiction in either case), as we now sketch. 

By the constraints on $\psi$ and the classification result of \cite{MaMiaoMurphyZheng2025}, the solution $w$ to \eqref{ECgNLS} with $w|_{t=0}=\psi$ either scatters in both time directions or equals the heteroclinic orbit $W^-$ modulo symmetries. Thus $w$ has finite $L_{t,x}^6$-norm either forward or backward in time (or both). Let $I=\R_+$ or $I=\R_-$ be an interval on which $w$ has finite $L_{t,x}^6$ norm. 

We then let $w_n$ be the solution to \eqref{ECgNLS} with $w_n(0)=P_{>\beta_n^\theta}\psi$, which (by stability theory for \eqref{ECgNLS}) also scatters on $I$ with uniform bounds for $n$ sufficiently large.  In particular, using persistence of regularity, we have
\[
\|\nabla w_n\|_{S(I)}\lesssim 1 \qtq{and} \|w_n\|_{S(I)}\lesssim \beta_n^{-\theta}. 
\]
The rescaled solutions $\tilde u_n(t,x)=\beta_n^{-1}w_n(\beta_n^{-2} t,\beta_n^{-1}x)$ to \eqref{ECgNLS} then satisfy
\[
\|\nabla \tilde u_n\|_{S(I)} \lesssim 1 \qtq{and}\|\tilde u_n\|_{S([0,\infty)} \lesssim \beta_n^{1-\theta}\to 0,
\]
so that
\[
\|\langle \nabla\rangle |\tilde u_n| \tilde u_n\|_{L_{t,x}^{\frac32}(I\times\R^4)} \lesssim \beta_n^{1-\theta}\to 0.
\]
Applying stability for \eqref{NLS}, we deduce the existence of solutions $z_n$ to \eqref{NLS} with $z_n(0)=\phi_n$ satisfying
\[
\|\langle\nabla\rangle z_n\|_{S(I)}\lesssim 1 \qtq{and} \lim_{n\to\infty} \|z_n\|_{L_{t,x}^3(I\times\R^4)} = 0. 
\]

We now let 
\[
U_n(t,x) = z_n(t,x) + e^{it\Delta} R_n.
\]
These functions obey global spacetime bounds and asymptotically solve \eqref{NLS} as $n\to\infty$ (cf. Step 2 in the proof of Proposition~\ref{P2}). Moreover, we have
\[
\|\langle\nabla\rangle e^{it\Delta}[U_n(0) - v_n(0)]\|_{L_{t,x}^3(I\times\R^4)} \lesssim \|e^{it\Delta} R_n\|_{L_{t,x}^3} + \|R_n\|_{\dot H^1} \to 0,
\]
which implies, by stability (see Lemma~\ref{stabi} and Remark~\eqref{R:stabi}), that
\[
\|v_n\|_{L_{t,x}^3\cap L_{t,x}^6(I\times\R^4)}\lesssim 1
\]
uniformly for large $n$. In particular, if $I=[0,\infty)$ this yields global spacetime bounds for the solutions $v_n$ on $(0,\infty)$, while if $I=(-\infty,0]$ we obtain bonds for $v_n$ on $(-\infty,0]$.  In either case, we obtain a contradiction to \eqref{vnbu}. 

We have therefore established that \eqref{cmcc1} holds, and thus (restoring the parameters $x_n$ for the sake of completeness) we have
\[
u_n(t_n,x-x_n) = \psi(x) + R_n(x),
\]
where $\psi\in H^1\backslash\{0\}$, $r_n\to 0$ strongly in $\dot H^1$ and weakly in $L^2$, and
\begin{equation}\label{mc-rn}
\lim_{n\to\infty} \| e^{it\Delta}R_n\|_{L_{t,x}^6\cap L_{t,x}^3} = 0.
\end{equation}

Now, using the properties of the profile decomposition, we have that
\[
m_c = M(u_n) + o(1) = M(\psi) + M(R_n) + o(1) \qtq{as}n\to\infty. 
\]
In particular, as $\psi\neq 0$, we obtain that $m_c\geq M(\psi)>0$. We also recall that $E(\psi)=E^c(W)$ and $\|\psi\|_{\dot H^1}<\|W\|_{\dot H^1}$.

We would now like to show that $M(\psi)=m_c$, so that $r_n\to 0$ in $L^2$ and hence $u_n(t_n,\cdot-x_n)\to \psi$ strongly in $H^1$.  Suppose instead that $M(\phi)<m_c$.  We then let $w_n$ be the solution to \eqref{NLS} with $w_n(0)=\phi(\cdot+x_n)$.  In light of the properties established above, we find that $w_n$ satisfies the energy constraint \eqref{Thres}. Using the definition of $m_c$ and the fact that $M(\phi)<m_c$, it follows that $w_n\notin\mathcal{B}$, i.e. $w_n$ scatters forward in time.  Moreover, by translation invariance, the scattering norms of $w_n$ are bounded uniformly in $n$. 

Recalling the solutions $v_n(t)=u_n(t+t_n)$, we now observe that
\[
\|w_n(0)-v_n(0)\|_{H^1} = \|R_n\|_{H^1} \lesssim 1
\]
and 
\[
\|\langle\nabla\rangle e^{it\Delta}[w_n(0)-v_n(0)]\|_{L_{t,x}^3} \lesssim \|e^{it\Delta} R_n\|_{L_{t,x}^3} + \|w_n(0)-v_n(0)\|_{\dot H^1} \to 0
\]
as $n\to\infty$. It follows from stability (see Lemma~\ref{stabi} and Remark~\eqref{R:stabi}) that 
\[
\|u_n\|_{L_{t,x}^3\cap L_{t,x}^6([t_n,\infty)\times\R^4)}=\|v_n\|_{L_{t,x}^3\cap L_{t,x}^3([0,\infty)\times\R^4)} \lesssim 1
\]
for all $n$ sufficiently large.  However, this contradicts the fact that $u_n\in\mathcal{B}$ for all $n$.

We therefore obtain that $M(\psi)=m_c$ and that $u_n(t_n,\cdot+x_n)$ converges strongly to $\psi$ in $H^1$.  We then let $u_c$ be the solution to \eqref{NLS} with $u_c(0)=\psi$.  This solution satisfies the energy constraint $\eqref{Thres}$ as well as the minimal mass condition $M(u_c)=m_c$.  Moreover, by stability, $u_c$ blows up in forward time, so that $u_c\in\mathcal{B}$.

\bigskip

\textbf{Part 2.} We turn to the next main step of the proof, in which we use the solution $u_c$ to deduce the existence of the solution obeying small modulation for all forward time.

By Proposition~\ref{P5}, we may choose a sequence $t_n\to\infty$ such that $\delta(u_c(t_n))\to 0$.  We may assume that $\delta(u_c(t_n))<\eta_*$ for all $n$.

Now, if 
\[
\exists t_0>0\qtq{such that}
\sup_{[t_0,\infty)}\delta(u_c(t))\leq \eta_*,
\]
then we can conclude the proof of Proposition~\ref{P6} immediately by using the solution $v\in \mathcal{B}$ defined by $v(t)=u_c(t+t_0)$. Thus it suffices to consider the alternate case.  

In this case, by passing to a subsequence, we can find a sequence $t_n^-\to\infty$ such that
\[
t_n^- < t_n < t_{n+1}^- \qtq{and} \delta(u_c(t_n^-))>\eta_*. 
\]

As $t\mapsto\delta(u_c(t))$ is continuous, we can then define
\begin{equation}\label{deftn+}
t_{n}^{+} = \inf \bigl\{ t \in \R : \sup_{\tau \in [t, t_{n}]} \delta(u_c(t))
< \eta_{\ast}  \big\},
\end{equation}
which satisfies $t_{n}^{+} \to\infty$ (as $t_n^+>t_n^-$) and
\begin{equation}\label{f4.2}
\delta(t_n^{+})= \eta_{\ast}.
\end{equation}

We now claim that because of the minimal mass property,
\begin{equation}\label{ucconverges}
u_c(t_n^+,x-x(t_n^+)) \qtq{converges strongly in}H^1\qtq{as}n\to\infty. 
\end{equation} 

Indeed, arguing as before, we can combine the profile decomposition with $\dot H^1$-compactness (together with \eqref{f4.2} and Proposition~\ref{P4}) to obtain a decomposition of the form
\[
u_c(t_n^+,x-x(t_n^+)) = \phi_n + R_n,
\]
where either 
\[
\phi_n\equiv \phi\qtq{for some}\phi\in H^1\backslash\{0\}
\]
or 
\begin{equation}\label{m20case}
\phi_n = \beta_n^{-1}[P_{>\beta_n^\theta}\phi](\tfrac{\cdot}{\beta_n}) \qtq{for some}\phi\in\dot H^1\backslash\{0\},\ \beta_n\to 0. 
\end{equation}
In either case, we have $R_n\to 0$ in $\dot H^1$, $R_n$ bounded in $L^2$, and 
\[
\lim_{n\to\infty} \|e^{it\Delta}R_n\|_{L_{t,x}^3\cap L_{t,x}^6}=0.
\]
Moreover we have $E(R_n)\to 0$.

Arguing as above, if \eqref{m20case} holds, then we may use approximation by suitable solutions to energy-critical NLS and stability theory to ultimately derive spacetime bounds for the solution $u_c$ on either $[0,t_n^+)$ or $[t_n^+,\infty)$, yielding a contradiction to the fact that $u_c\in\mathcal{B}$.  Thus we conclude that $\phi_n\equiv \phi$.

It now remains to prove that $R_n\to 0$ in $L^2$.  To this end, we observe that 
\[
E(\phi)\leq E(u_c)=E^c(W)\qtq{and}\| \phi\|_{\dot H^1}^2\leq \| u_c(t_n^+)\|_{\dot H^1}^2\leq\| W\|_{\dot H^1}^2-\eta_*.
\]
We claim that $M(\phi)=M(u_c)=m_c$ , so that we obtain the desired convergence in $H^1$.  Suppose instead that $M(\phi)<m_c$. Then, since
\[
E(\phi)\leq E^c(W),\quad \|\nabla \phi\|_{L^2}<\|\nabla W\|_{L^2},\qtq{and} M(\phi)<m_c,
\]
it follows that the solutions $v_n$ to \eqref{NLS} with $v_n(t_n)=\phi(\cdot+x(t_n^+))$ are global and scatter (with spacetime norms bounded uniformly in $n$).

We now observe
\[
\|u_c(t_n) - v_n(t_n)\|_{H^1} = \|R_n\|_{H^1} \lesssim 1
\]
and
\begin{align*}
\|\langle \nabla\rangle e^{i(t-t_n)\Delta}[u_c(t_n)&-v_n(t_n)]\|_{L_{t,x}^3([t_n,\infty)\times\R^4)} \\
&\lesssim \|e^{it\Delta} R_n\|_{L_{t,x}^3} + \|R_n\|_{\dot H^1} \to 0.
\end{align*}
Thus by stability we deduce spacetime bounds for $u_c$ on $[t_n^+,\infty)$, a contradiction. 

\bigskip

\textbf{Part 3.} Finally, let $u$ be the solution to \eqref{NLS} with  $u|_{t=0}=\phi$. Then $u$ satisfies the energy constraint \eqref{Thres} and the mass constraint $M(u)=M(u_c)$.  Moreover, using stability together with the facts that $u_c$ blows up forward in time and $t_n^+\to\infty$, we have that $u$ blows up in both time directions. In particular, $u\in\mathcal{B}$ (and hence is forward global by Proposition~\ref{P2}).

We now claim that
\begin{equation}\label{final_claim}
 \sup_{t\in[0,\infty)} \delta(u(t))\leq \eta_*,
\end{equation}
so that $u$ yields the desired solution with small modulation for all forward time. Note that $\delta(u(0))\leq\eta_*$ by construction. To prove \eqref{final_claim}, we claim that it will suffice to prove that for fixed $T>0$, 
\begin{equation}\label{finalz}
[0,T]\subset [0, t_n-t_n^+]\ \qtq{for sufficiently large}n.
\end{equation}
Indeed, fixing $T$ and supposing \eqref{finalz} holds, we define 
\[
\tilde{u}_n(t,x):=u_c(t+t_n^+,x-x(t_n^+))
\]
and use a change of variables, the definition \eqref{deftn+}, and stability to obtain
\begin{align*}
\delta(u(T)) & \leq \delta(\tilde u_n(T)) + \|\tilde u_n(T)\|_{\dot H^1}^2 - \|u(T)\|_{\dot H^1}^2 \\
& \leq \delta(u(t_n^++ T)) + C\|\tilde u_n(T)-u(T)\|_{\dot H^1} \\
& \leq \eta_* + o(1)\qtq{as}n\to\infty.
\end{align*}
This indeed yields \eqref{final_claim} and completes the proof of Proposition~\ref{P6}.

We turn to the proof of \eqref{finalz}.  Recalling that that $\|\tilde{u}_n(0)-u(0)\|_{H^1}\to 0$, stability yields
\[
\max_{t\in[0,T]}\|\tilde{u}_n(t)-u(t)\|_{H^1}\to 0,
\]
so that
\begin{equation}\label{finalzz}
\max_{t\in[0,T]}|\delta(\tilde{u}_n(t))-\delta(u(t))|\to 0.
\end{equation}

Now suppose that \eqref{finalz} fails, so that $[0,t_n-t_n^+]\subset[0,T]$.  Passing to a subsequence, we then have $t_n-t_n^+\to t_0\in[0,T]$. As
\[
\delta(\tilde u_n(t_n-t_n^+)) = \delta(u_c(t_n))\to 0,
\]
we obtain from \eqref{finalzz} that
\[
|\delta(u(t_n-t_n^+))| \leq |\delta(\tilde u_n(t_n-t_n^+))| + \max_{t\in[0,T]}|\delta(\tilde{u}_n(t))-\delta(u(t))| \to 0. 
\]
However, by continuity of the flow, this implies that $\delta(u(t_0))=0$, which contradicts \eqref{PositiveP}. Thus we finally complete the proof of \eqref{finalz} and Proposition~\ref{P6}. \end{proof}


\section{Preclusion of small modulation solutions}\label{S:preclude}

In this section we prove Proposition~\ref{P7}. 

We first record the following lemma, which is a direct consequence of the modulation analysis (see Lemma~\ref{modula}).  We also recall the notation \eqref{mod-notation}.  We remark throughout in this section we use $\lambda(\cdot)$ and $x(\cdot)$ to denote the modulation parameters of the solution, whereas in previous sections these denoted compactness parameters.

\begin{lemma}[Modulation]\label{Modulation} Suppose $u\in\mathcal{B}$ satisfies 
\[
\sup_{t\in[0,\infty)}\delta(u(t))\leq \eta_*
\]
for sufficiently small $\eta_*>0$. Then there exist continuous functions
\begin{equation*}
\theta:[0,\infty)\rightarrow \R,\quad x:  [0,\infty)\rightarrow \mathbb{R}^4 ,\quad \lambda: [0,\infty)\rightarrow \mathbb{R}_+, \quad \alpha: [0,\infty)\rightarrow \mathbb{R}_+
\end{equation*}
such that if 
\begin{equation}\label{decompzz}
    v(t):=\alpha(t)W+h(t)=u_{[\theta(t),\lambda(t),x(t)]}(t)-W,
\end{equation}
then there exists $C>0$ such that
\begin{equation}\label{mo1z}
|\alpha(t)|\sim \|v(t)\|_{\dot{H}^1}\sim \|h(t)\|_{\dot{H}^1}\sim \delta(u(t)), 
\end{equation}
\begin{equation}\label{mo3z}
\|u(t)\|_{L_x^{3}}^{3}\leq C\delta(u(t))^2,
\end{equation}
\begin{equation}\label{mo2z}
\bigl|x'(t) - \tfrac{\lambda'(t)}{\lambda(t)}x(t)\bigr| +|\alpha'(t)|+|\theta'(t)|+\bigl|\tfrac{\lambda'(t)}{\lambda(t)}\bigr|\leq C\lambda(t)^2 \delta(u(t)),
\end{equation}
\begin{equation}\label{mo4z}
\frac{\lambda(t)^{\frac32}}{[1+\lambda(t)^2]^{\frac32}}\leq  C\delta(u(t)).
\end{equation}
\end{lemma}

We will also need the following lemma establishing a lower bound for the virial quantity in terms of $\delta(u)$. We recall the notation from Section~\ref{S:virial}.

\begin{lemma}\label{Compa} Let $u$ be an $H^1$ solution to \eqref{NLS} satisfying \eqref{Thres}. Then there exists $c_0>0$ such that
\begin{equation}\label{InequeN}
{{F^{c}_{\infty}}}[u(t)]=8\|\nabla u(t)\|^{2}_{L^{2}}-8\|u(t)\|^{4}_{L^{4}}\geq c_0\delta(u(t))
\qtq{for all $t\geq 0$.}
\end{equation}
\end{lemma}
\begin{proof}
We recall that $E^{c}(W)=\frac{1}{4}\|\nabla W\|^{2}_{L^{2}}$.  By using $E(u)=E^{c}(W)$ we have 
\begin{equation}\label{Pohi22}
{{F^{c}_{\infty}}}[u(t)]=8\delta(u(t))-\tfrac{32}{3}\|u(t)\|^{3}_{L^{3}}.
\end{equation}
We claim that ${{F^{c}_{\infty}}}[u(t)]>0$ for all $t\geq 0$. Indeed, by the sharp Sobolev inequality \eqref{GI} and \eqref{PositiveP},
\begin{align}\nonumber
\|u\|^{4}_{L^{4}}&\leq  C^{4}_4\|\nabla u\|_{L^2}^4
= \tfrac{1}{\|\nabla W\|^{4}_{L^{2}}}\| \nabla u\|^{4}_{L^{4}}\\ \label{IGU}
&=\bigl[\tfrac{\|\nabla u\|_{L^2}}{\|\nabla W\|_{L^{2}}}\bigr]^{2}\|\nabla u\|^{2}_{L^{2}}
<\|\nabla u\|^{2}_{L^{2}}.
\end{align}

Suppose now that \eqref{InequeN} were false. Then there exists  $\left\{t_{n}\right\}$ such that
\begin{equation}\label{Contra22}
{{F^{c}_{\infty}}}[u(t_{n})]\leq \tfrac{1}{n}\delta(u(t_{n})).
\end{equation}
We first observe that $\left\{\delta(u(t_{n}))\right\}$ is bounded (cf. Lemma~\ref{GlobalS}), so that $F^c_\infty[u(t_n)]\to 0$. Thus \eqref{IGU} implies
\begin{align*}
\|\nabla u(t_n)\|_{L^2} \to \|\nabla W\|_{L^2},\qtq{i.e.}\delta(u(t_n))\to 0.
\end{align*}
Now, by Lemma~\ref{Modulation} (or Lemma~\ref{modula}) we have
\[
\tfrac{32}{3}\| u(t_n)\|^{3}_{L^{3}}\lesssim \delta(u(t_{n}))^{2}\leq  4\delta(u(t_{n}))\quad \text{for $n$ large}.
\]
Thus, by using \eqref{Pohi22} and \eqref{Contra22} we get
\[
4\delta(u(t_{n}))\leq \tfrac{1}{n}\delta(u(t_{n}))\quad \text{for $n$ large},
\]
contradicting the fact that $\delta(u(t))>0$ for $t>0$.\end{proof}

We now turn to the proof of Proposition~\ref{P7}.

\begin{proof}[Proof of Proposition~\ref{P7}]

We suppose towards a contradiction that $u\in\mathcal{B}$ is a solution to \eqref{NLS} satisfying $\sup_{t\in[0,\infty)}\delta(u(t))\leq \eta_*$. 

\bigskip

\textbf{Part 1.} We begin by proving an almost constancy property for the frequency scale and spatial center of $u$ (cf. Proposition~\ref{P2}).  By translation invariance, we may assume $x(0)=0$.  The proof is based on a bootstrap argument using the localized virial estimate incorporating modulation.

We wish to prove the following: 

\emph{On any interval $[0,T]$ such that}
\begin{equation}\label{bootstrap-assumption}
\max_{t\in[0, T]}\lambda(t)\leq2\min_{t\in[0, T]}\lambda(t)\qtq{and} \max_{t\in[0, T]}|{x(t)}|\leq 1,
\end{equation}
\emph{we have}
\begin{equation}
\max_{t\in[0, T]}\lambda(t)\leq\tfrac{3}{2}\min_{t\in[0, T]}\lambda(t),\quad  \max\limits_{t\in[0, T]}|{x(t)}|\leq \tfrac{1}{2}.
\end{equation}

\emph{Consequently,}
\begin{equation}\label{bootse1}
\max_{t\in[0, \infty)}\lambda(t)\leq\tfrac{3}{2}\min_{t\in[0, \infty)}\lambda(t),\quad  \max_{t\in[0, \infty)}\left|{x(t)}\right|\leq \tfrac{1}{2}.
\end{equation}

We turn to the proof. We suppose $[0,T]$ is an interval on which \eqref{bootstrap-assumption} holds. We introduce the notation
\[
\lambda_{\min}=\min_{t\in[0,T]}\lambda(t) \qtq{and} \lambda_{\max} = \max_{t\in[0,T]}\lambda(t). 
\]
We let $\varphi$ be a radial function satisfying
\begin{equation}\label{start1}
\varphi(x) = \begin{cases} |x|^2, \qquad & | x | \le 1, \\ 0, \qquad & | x | \ge 2.  
\end{cases} \qtq{with} \left|\partial^{\alpha}\varphi(x)\right|\lesssim_{\alpha} |x|^{2-|\alpha|} 
\end{equation}
for all multiindices $\alpha$. We also set
\[
\varphi_{R}(x)=R^2\varphi(\tfrac{x}{R}).
\]
We now define the Morawetz potential
\[
I_{R}[u]=2\Im \int_{\R^4} \nabla \varphi_R(x) \cdot \overline{u(t,x)}\nabla u(t,x)\, dx.
\]

We will first show that for any $R\geq 1$, we have
\begin{equation}\label{step1}
|M_R(u(t))| \lesssim R^2\delta(u(t)).
\end{equation}
We will then show that if we take $R$ to be
\[
R=\max_{t\in[0,T]} \tfrac{C}{\lambda(t)} = C\lambda_{\min}^{-1}
\]
for some large (universal) $C\gg1$, then we can obtain the lower bound
\begin{equation}\label{step2}
\tfrac{d}{dt}M_R(u(t)) \geq c\delta(u(t))
\end{equation}
for some constant $c$ that independent of $\eta_{\ast}$.

Combining \eqref{step1} and \eqref{step2} and applying the Fundamental Theorem of Calculus then yields
\[
\int_0^T \delta(u(t)) \,dt \lesssim \lambda_{\min}^{-2}[\delta(u(0))+\delta(u(T))]\lesssim \eta_*\lambda_{\min}^{-2} .
\]
Applying the bound on $|\tfrac{\lambda'(t)}{\lambda(t)}|$ in \eqref{mo2z} and \eqref{bootstrap-assumption}, this yields
\[
\int_0^T \bigl|\tfrac{d}{dt}\log \lambda(t)\bigr|\,dt \lesssim \eta_* \lambda_{\min}^{-2}\lambda_{\max}^{2} \lesssim \eta_*. 
\]
Choosing $\eta_*$ sufficiently small (depending only on universal constants) yields 
\[
\lambda_{\max}\leq\tfrac32\lambda_{\min}.
\]

Similarly, applying the upper bound on $|x'-\tfrac{\lambda'}{\lambda} x|$ and \eqref{bootstrap-assumption}, we obtain
\begin{align*}
\int_0^T |x'(t)|\,dt & \leq \int_0^T |x'(t)-\tfrac{\lambda'(t)}{\lambda(t)}x(t)|\,dt + \int_0^T |\tfrac{\lambda'(t)}{\lambda(t)}|\, |x(t)|\,dt  \\
& \lesssim \lambda_{\max}^2 \lambda_{\min}^{-2} \eta_* \lesssim \eta_*,
\end{align*}
which (choosing $\eta_*$ sufficiently small) implies 
\[
\max_{t\in[0,T]}|x(t)| \leq \tfrac12.
\]

Thus it remains to prove \eqref{step1} and \eqref{step2}.  In what follows we use the notation
\[
\tilde W(t):=\lambda(t) e^{-i\theta(t)}W(\lambda(t)x+x(t)).
\]

We first prove \eqref{step1}.  First, recalling that $W$ is real-valued, we note that
\[
I_{R}[\tilde W(t)]\equiv 0.
\]
Using H\"older's inequality and Sobolev embedding, and the fact that
\[
\| u\|_{L_t^\infty \dot H_x^1} \leq \|W\|_{\dot H_x^1},
\]
 we estimate
\begin{align*}
|I_{R}[u(t)]| & = |I_{R}[u(t)] - I_{R}[\tilde W(t)]| \\
& \lesssim R\int_{|x|\lesssim R} [|\nabla u| |u| - |\nabla \tilde W| |\tilde W|]\,dx \\
& \lesssim R^2 \{\|W\|_{\dot H^1}+\| u\|_{L_t^\infty \dot H_x^1}\}\{\|u-\tilde W\|_{L_t^\infty \dot H_x^1} + \|u-\tilde W\|_{L_t^\infty L_x^4}\} \\
& \lesssim R^2 \|u-\tilde W\|_{L_t^\infty \dot H_x^1}\lesssim R^2 \delta(u(t)). 
\end{align*}

We turn to the proof of \eqref{step2}. By Lemma \ref{VirialIden} and Lemma \ref{Compa},
\begin{align*}
\tfrac{d}{dt} I_{R}[u] & = \int_{\R^4}\big\{(- \Delta \Delta w_{R})|u|^{2}-
				\Delta[w_{R}]|u|^{4}+4\Re \overline{u_{j}} u_{k} \partial_{jk}[w_{R}]\big\}\,dx\\
				&\quad +\tfrac{2}{3}\int_{\R^4}\Delta[w_{R}]|u|^{3}\,dx\\
& =F^{c}_{R}[u(t)]+\tfrac{2}{3}\int_{\R^4}\Delta[w_{R}]|u|^{3}\,dx\\
& =F^{c}_{R}[u(t)]-F_{R}^c(\tilde{W})+\tfrac{2}{3}\int_{\R^4}\Delta[w_{R}]|u|^{3}\,dx\\
& =F^{c}_{\infty}[u(t)] +\{(F_{R}^c[u(t)]-F^{c}_{\infty}[u(t)])-(F_{R}^c[\tilde{W}(t)]-F^{c}_{\infty}[\tilde{W}(t))]\}\\
&\quad +\tfrac{2}{3}\int_{\R^4}\Delta[w_{R}]|u|^{3}\,dx\\
& \geq c_0\delta(u(t))+\mathcal{G}_{R}[u(t)]-\mathcal{G}_{R}[\tilde W(t)] +\tfrac{2}{3}\int_{\R^4}\Delta[w_{R}]|u|^{3}\,dx,
\end{align*}
where 
\begin{align*}
\mathcal{G}_{R}[u] & := -8\int_{|x|\geq R}|\nabla u|^2+8\int_{|x|\geq R}|u|^4\,dx \\
& \quad +\int_{R\leq|x|\leq 2R} [-\Delta\Delta\varphi_R]|u|^2 \,dx - \int_{|x|\geq R} \Delta[\varphi_R]|u|^4 \, dx\\
& \quad + 4\Re\int_{|x|\geq R} \bar u_j u_k \partial_{jk}[\varphi_R]\,dx.
\end{align*}

We now wish to show that each term appearing in this difference can be controlled by a small multiple of $\delta(t)$, provided $R$ is chosen appropriately.  

Let $\eta>0$.  We begin by changing variables to obtain
\begin{align*}
\biggl|\int_{|x|\geq R} & |\nabla u(t)|^2 - |\nabla \tilde W(t)|^2 \,dx\biggr| \\
& \lesssim\bigl\{ \|\nabla u(t)\|_{L^2(|x|\geq R)} + \|\nabla W\|_{L^2(|x-x(t)|\geq R\lambda(t)} \bigr\} \|\nabla v(t)\|_{L^2}.
\end{align*}
We now note that 
\[
\{|x-x(t)|>R\lambda(t)\} \subset \{|x|\geq \tfrac12 C\}.  
\]
Indeed, this follows from the fact that $|x(t)|\leq 1$ and the definition of $R$. Thus we can choose $C$ sufficiently large to obtain
\[
\|\nabla W\|_{L^2(|x-x(t)|\geq R\lambda(t))}\leq \eta. 
\]
On the other hand, recalling \eqref{decompzz} and changing variables,
\[
\|\nabla u\|_{L^2(|x|\geq R)} \leq \|\nabla W\|_{L^2(|x-x(t)|\geq R\lambda(t))}+\|\nabla v(t)\|_{L_x^2} \lesssim \eta+\delta(u(t)).
\]
Thus 
\[
\biggl|\int_{|x|\geq R} |\nabla u(t)|^2 - |\nabla \tilde W(t)|^2 \,dx\biggr| \lesssim [\eta+\eta_*] \delta(u(t)). 
\]

The remaining terms may be treated in a similar manner. As an example, let us demonstrate how to estimate one additional term.  Using H\"older's inequality and Sobolev embedding, 
\begin{align*}
\biggl| \int_{|x|\geq R} & |u|^4 \,dx - \int_{|x|\geq R} |\tilde W|^4\,dx \biggr| \\
 \lesssim& \bigl\{ \|u\|_{L_x^4(|x|\geq R)}^3 + \|W\|_{L_x^4(|x-x(t)|\geq R\lambda(t))}^3\} \|v(t)\|_{L_x^4} \\
 \lesssim &[\eta^3+\eta_*^3]\|v(t)\|_{\dot H^1} \lesssim [\eta^3+\eta_*^3]\delta(u(t)). 
\end{align*}
Estimating similarly for the remaining terms and using properties of $\varphi$, we obtain
\[
|\mathcal{G}_{R}[u(t)]-\mathcal{G}_{R}[\tilde W(t)]| \lesssim [\eta+\eta_*]\delta(u(t)).
\]

Finally, using $|\Delta[w_{R}]|\lesssim1$ and \eqref{mo3z}, we estimate
\[
\left|\int_{\R^4}\Delta[w_{R}]|u|^{3}\, dx\right|\lesssim \delta(u(t))^2\lesssim \eta_{\ast}\delta(u(t)).
\]

Taking $\eta>0$ and $\eta_{\ast}>0$ sufficiently small, we complete the proof of \eqref{step2} and hence complete the bootstrap argument.

\bigskip

\textbf{Part 2.} We are now in a position to derive a contradiction and complete the proof of Proposition~\ref{P7}.  Recalling that $x(0)=0$ and using \eqref{bootse1}, we have
 \begin{equation*}
 {\lambda}(t)\sim {\lambda}(0)\qtq{and} |x(t)|\leq 1, \qtq{for all}t>0.
 \end{equation*}
Now, by Proposition~\ref{P5}, there exists $t_n\to\infty$ so that $\delta(u(t_n))\to 0$. Thus by \eqref{mo4z},
\begin{equation*}
\frac{\lambda(t_n)^{3/2}}{(1+\lambda(t_n)^2)^{3/2}}\lesssim \delta(u(t_n))\to 0,
\end{equation*}
contradicting $\lambda(t_n)\sim\lambda(0)$.\end{proof}

\appendix

\section{Modulation analysis}\label{S:Modulation}

In this section we assume that $u:I\times\R^4\to\C$ is a solution to \eqref{NLS} satisfying the energy constraint \eqref{Thres}.  We recall the definition of $\delta$ given in \eqref{deltau}.  By Lemma~\ref{GlobalS}, we have that $\delta(u(t))>0$ for all $t\in I$. 

For $(\theta_0, \mu_0, x_0) \in \mathbb{R} \times (0,\infty) \times \mathbb{R}^4$,  we use the notation
\begin{equation}\label{mod-notation}
f_{[\theta_{0}, \mu_{0}, x_{0}]}(x) := e^{i\theta_0} \mu_0^{-1} f(\mu_0^{-1}(x - x_{0})).
\end{equation}
For $\delta_0 > 0$ small (to be determined below), we define
\[
I_0 = \{t \in I : \delta(u(t)) < \delta_0\}.
\]
\begin{lemma}\label{modula} 
There exists $\delta_0>0$ sufficiently small such that if $u$ satisfies $\delta(u(t))<\delta_0$ on the interval $I_0$, then there exist continuous functions
\begin{equation*}
\theta(t) : I_0\rightarrow \R,\quad x(t): I_0\rightarrow \mathbb{R}^4,\quad \mu(t): I_0\rightarrow \mathbb{R}_+,\quad \alpha(t): I_0\rightarrow \mathbb{R}_+
\end{equation*}
such that if 
\begin{equation}\label{decompz}
v(t):=\alpha(t)W+h(t)=u_{[\theta(t),\mu(t),x(t)]}(t)-W,
\end{equation}
then
\begin{equation}\label{mo1}
|\alpha(t)|\sim \|v(t)\|_{\dot{H}^1}\sim \|h(t)\|_{\dot{H}^1}\sim \delta(u(t)), 
\end{equation}
\begin{equation}\label{mo2}
\|u(t)\|_{L_x^{3}}^{3}\lesssim \delta(u(t))^2,
\end{equation}
\begin{equation}\label{mo3}
\bigl|x'(t) - \tfrac{\mu'(t)}{\mu(t)}x(t)\bigr| +|\alpha'(t)|+|\theta'(t)|+\bigl|\tfrac{\mu'(t)}{\mu(t)}\bigr|\lesssim \mu(t)^2 \delta(u(t)),
\end{equation}
\begin{equation}\label{mo4}
\frac{\mu(t)^{\frac{3}{2}}}{(1+\mu(t)^2)^{\frac32}}\lesssim \delta(u(t)),
\end{equation}
with implicit constants independent of $u$.
\end{lemma}

We begin with the following result, which follows from the arguments used to prove \cite[Lemma 3.2]{AJZ}.		
\begin{lemma}\label{aux1L}
For any $\varepsilon > 0$, there exists $\delta_0 = \delta_0(\varepsilon) > 0$ small enough that if $\delta(u(t)) < \delta_0$, then there exists $(\theta_0(t), \mu_0(t), x_0(t)) \in \mathbb{R} \times (0,\infty) \times \mathbb{R}^4$ so that
\begin{align}\label{asc}
\left\| u_{[\theta_0(t),\mu_0(t),x_0(t)]}(t, \cdot) - W \right\|_{\dot{H}^1} < \varepsilon. 
\end{align}
\end{lemma}

With Lemma~\ref{aux1L} in hand, we can next follow the arguments used to prove \cite[Lemma 5.1]{CamposFarahRoudenko} to obtain the following.
	
\begin{lemma}\label{Lemmaes}
If $\delta_0 > 0$ is sufficiently small, then there exist $C^1$ functions $\theta(t) : I_0 \to \mathbb{R}$, $\mu(t) : I_0 \to (0,\infty)$ and $x(t): I_0 \to \mathbb{R}^{4}$ such that
\[
\|u_{[\theta(t),\mu(t),x(t)]}(t) - W\|_{\dot{H}^1} \ll 1.
\]
Moreover, if we define $g(t) = u_{[\theta(t),\mu(t),x(t)]}(t)$, then we have
\[
g(t) \perp \operatorname{span}\{\nabla W, iW, \Lambda W\}\qtq{for all $t\in I_{0}$,}
\]
where  $\Lambda W=W + x \cdot \nabla W$.
\end{lemma}

\begin{remark}\label{Rem1}
Let $R \geq 1$. We claim that $\delta_0$ sufficiently small (depending on $u$), $\mu(t) \geq R$ for all $t \in I_0$. If not, there exist $t_n$ with $\delta(u(t_n)) \to 0$ and $\mu(t_n) \leq R$. Then we can define
\[
f_n(x) := u_{[\theta(t_n), \mu(t_n), x(t_n)]}(t_n, x).
\]
By \eqref{asc}, $f_n \to W$ in $\dot{H}^1$. On the other hand, mass conservation gives 
\[
\|f_n\|_{L^2} = \mu(t_n) M(u_0)^{1/2} \leq R M(u_0)^{1/2}.
\]
Thus $f_n$ also converges weakly in $L^2$ along a subsequence, contradicting that $W\notin L^2$.
\end{remark}

We now introduce the operators
\begin{align*}
L_R = -\Delta - 3W^2\qtq{and} L_I = -\Delta - W^2.
\end{align*}
	
Associated with these operators we define the quadratic form\footnote{We will write complex-valued functions in the form $f=f_1+if_2$, where $f_1$ is the real part and $f_2$ the imaginary part.}
\begin{align}
\mathcal{F}(f) &= \tfrac{1}{2}(L_R f_1, f_1) + \tfrac{1}{2}(L_I f_2, f_2) \\
&= \tfrac{1}{2} \int |\nabla f|^2  \, dx - \tfrac{1}{2} \int W^2(3|f_1|^2 + |f_2|^2)  \, dx.
\end{align}

Denote by $G := \text{span}\{W, \nabla W, iW, \Lambda W\}$ and by $G^\perp$ its orthogonal complement in $\dot{H}^1$ (viewed as a subspace of $\dot{H}^1$). Then \cite[Lemma 3.5]{CamposFarahRoudenko} provides the following lemma:

\begin{lemma}\label{CoecTy}
There exists $C > 0$ such that for  $h \in G^\perp$,
\[
\mathcal{F}(h) \geq C\|h\|_{\dot{H}^1}^2.
\]
\end{lemma}

We now write
\begin{equation}\label{Decop1}
u_{[\theta(t),\mu(t),x(t)]}(t) = (1 + \alpha(t))W + h(t), 
\end{equation}
where
\[
\alpha(t) = \tfrac{1}{\|W\|_{\dot{H}^1}} (u_{[\theta(t),\mu(t),x(t)]}, W)_{\dot{H}^1} - 1.
\]
In particular, $\alpha(t)$ is chosen so that $h(t) \in G^\perp$.

Using the decomposition \eqref{decompz}, the fact that $(E^{c})'(W) = 0$, and Taylor series expansion, we have that
\begin{equation}\label{IgD}
E^{c}(u(t))-E^{c}(W)= \mathcal{F}(v(t)) + o(\|v\|_{\dot{H}^1}^2).
\end{equation}

We next establish the following lemma concerning the part of the potential energy coming from the mass-critical nonlinear term. 

\begin{lemma}\label{Lce1}
Let $(\theta(t), \mu(t), x(t))$  be as in Lemma~\ref{Lemmaes}. Then
\begin{equation}\label{Estn1}
\int_{\mathbb{R}^4} |u(t, x)|^3  \, dx \lesssim \delta^2(t) \sim \|v(t)\|_{\dot{H}^1}^2.
\end{equation}
\end{lemma}

\begin{proof} Since $E(u(t)) = E^{c}(W)$, equation \eqref{IgD} yields
\begin{equation}\label{ECS}
0 = \mathcal{F}(v(t)) + \tfrac{1}{3}\|u(t)\|_{L^3}^3 + o(\|v\|_{\dot{H}^1}^2).
\end{equation}
	
By Lemma~\ref{Lemmaes} we have the estimate
\begin{equation}\label{EstimaA}
|\alpha(t)| \lesssim \|v\|_{\dot{H}^1} \ll 1.
\end{equation}
	
Since $h \in G^\perp$, Lemma~\ref{CoecTy} implies
\[
\mathcal{F}(h(t)) \gtrsim \|h(t)\|_{\dot{H}^1}^2.
\]
Combining this with \eqref{ECS} gives
\[
\|h\|_{\dot{H}^1}^2 + \|u\|_{L^3}^3 \lesssim \alpha^2 + |\alpha( L_R W, h)| + o(\|v\|_{\dot{H}^1}^2).
\]
As $L_R W = 2\Lambda W$ and $h \in G^\perp$, we have $( L_R W, h) = -2(W, h)_{\dot{H}^1} = 0$, so
\[
\|h\|_{\dot{H}^1}^2 + \|u\|_{L^3}^3 \lesssim \alpha^2 + o(\|h\|_{\dot{H}^1}^2).
\]
In particular,
\begin{equation}\label{halp}
\|h(t)\|_{\dot{H}^1}^2 \lesssim \alpha^2(t) \quad \text{and} \quad \|u(t)\|_{L^3}^3 \lesssim \alpha^2(t).
\end{equation}
	
The orthogonality $(W, h)_{\dot{H}^1} = 0$ implies
\[
\|v\|_{\dot{H}^1}^2 = \alpha^2\|W\|_{\dot{H}^1}^2 + \|h\|_{\dot{H}^1}^2,
\]
so by \eqref{EstimaA} and \eqref{halp}, $\|v\|_{\dot{H}^1} \sim |\alpha|$. Finally, since $(W, h)_{\dot{H}^1} = 0$,
\[
\delta(t) = \left| \|W\|_{\dot{H}^1}^2 - \|(1 + \alpha)W + h\|_{\dot{H}^1}^2 \right| = 2|\alpha|\|W\|_{\dot{H}^1}^2 + O(\alpha^2),
\]
showing $\delta \sim |\alpha|$. Combining these estimates yields
\[
\|u(t)\|_{L^3}^3 \lesssim \alpha^2(t) \sim \delta^2(t) \sim \|v(t)\|_{\dot{H}^1}^2,
\]
which completes the proof. \end{proof}

We next prove the estimate concerning $\mu(t)$.

\begin{lemma}\label{Ders}
	Under the conditions of Lemma~\ref{Lce1}, if $\delta_0$ is sufficiently small, then there exists a positive constant $C$ such that
	\begin{equation}\label{actr}
			\frac{\mu(t)^{\frac{3}{2}}}{(1+\mu(t)^2)^{\frac32}}\leq C\delta(u(t)).
		\end{equation}
\end{lemma}

\begin{proof}
Since $W(x) = (1 + \tfrac{1}{8}|x|^2)^{-1}$,  we obtain
\begin{align}\label{deqa}
\mu(t)W(\mu(t)x+x(t)) &\geq \mu(t)(1 + \tfrac{1}{8}\mu(t)^2|x+\mu(t)^{-1}x(t)|^2)^{-1} \nonumber \\
&\gtrsim \frac{\mu(t)}{1+\mu(t)^2}
\end{align}
for all $|x+\mu^{-1}(t)x(t)|\leq 1$ and $t\in I_0$. 
	
On the other hand, by H\"older's inequality and Sobolev embedding, we have
\begin{align*}
\|u(t)& - W_{[-\theta(t),\mu^{-1}(t),-\mu(t)^{-1}x(t)]}\|_{L^3(B(-\mu^{-1}(t)x(t),1))}\\
&= \|v_{[-\theta(t),\mu^{-1}(t),-\mu^{-1}(t)x(t)]}(t)\|_{L^3(B(-\mu^{-1}(t)x(t),1))} \\
&\lesssim \|v(t)\|_{L^4} \lesssim \|v(t)\|_{\dot{H}^1} \lesssim \delta(t).
\end{align*}
From \eqref{Estn1} we deduce
\[
\|\mu(t)W(\mu(t)x+x(t))\|_{L^3(B(-\mu^{-1}(t)x(t),1))} \lesssim \delta(t) + \delta(t)^{\frac23} \lesssim \delta(t)^{\frac23}.
\]
Combining \eqref{deqa} with the above inequality now yields \eqref{actr}.
\end{proof}

We now prove the remaining estimates on the modulation parameters.

 \begin{lemma}\label{Molfa}
Under the conditions of Lemma~\ref{Lce1}, the following estimate holds:
\begin{equation}\label{modu4}
\bigl|x'(t) - \tfrac{\mu'(t)}{\mu(t)}x(t)\bigr| + |\alpha'(t)| + |\theta'(t)| + \bigl|\tfrac{\mu'(t)}{\mu(t)}\bigr| \leq C\mu(t)^2 \delta(u(t))
\end{equation}
 for some absolute constant $C>0$.
 \end{lemma}
 
 \begin{proof}
 	
 	We begin by introducing the following notation:
\[
\begin{aligned}
 &    W(t) := W(t,x) = W_{[-\theta(t),\mu^{-1}(t),-\mu^{-1}(t)x(t)]}(x), \\
 &    [\Lambda W](t) = (\Lambda W)_{[-\theta(t),\mu^{-1}(t),-\mu^{-1}(t)x(t)]}, \\
 &    [\partial_{j} W](t) = (\partial_{j} W)_{[-\theta(t),\mu^{-1}(t),-\mu^{-1}(t)x(t)]}, \quad j=1,\dots,4\\
 &    H(t, x)= h_{[-\theta(t),\mu^{-1}(t),-\mu^{-1}(t)x(t)]}(t,x).
 \end{aligned}
 \]
 Note that 
 \[
 W(t,x) = e^{-i\theta(t)} \mu(t) W\bigl(\mu(t)x + x(t)\bigr).
 \]
 From the decomposition given in \eqref{decompz}, we have:
 \[
 H(t,x) = u(t,x) - (1+\alpha(t))W(t).
 \]
 	
 Performing the change of variable $t = t(s)$ such that $\frac{dt}{ds} = \frac{1}{\mu(t)^2}$, a direct computation shows that:
\begin{align}\label{PWs}
\partial_s W(s,x) = -i\theta'(s)W(s,x) + \tfrac{\mu'(s)}{\mu(s)} \Lambda W(s,x)  + \big[x'(s)-\tfrac{\mu'(s)}{\mu(s)}x(s)\big]\cdot\nabla W(s,x)
 \end{align}
as well as
\begin{align}\label{PWs22}
&\partial_s H = \tfrac{1}{\mu(s)^{2}} \partial_t u - \alpha'(s)W(s,x) - (1+\alpha(s))\partial_s W(s,x), \\ \label{PWs33}
&\Delta H(s,x) = \Delta u(s,x) - (1+\alpha(s))[- |W(s,x)|^2 W(s,x)].
 \end{align}

Combining equations \eqref{PWs}, \eqref{PWs22} and \eqref{PWs33}, and recalling that $u(s)$ satisfies equation \eqref{NLS}, we deduce:
\begin{align}\label{EwsE}
i\partial_s H + \tfrac{1}{\mu(s)^2} \Delta H = & \tfrac{1}{\mu(s)^2} \left[(|W(s)|^2 W(s) - |u|^2 u) + |u| u + \alpha(s) |W(s)|^2 W(s) \right] \nonumber \\
& - i\alpha'(s) W(s) - i(1+\alpha(s)) \partial_s W(s).
\end{align}
 	
Now, by H\"older's inequality, interpolation, and conservation of mass:
\begin{align*}
|(|u| u, \phi)_{\dot{H}^1}| &= |(|u| u, \Delta\phi)_{L^2}| \\
&\leq \|u\|_{L^3} \|u\|_{L^3}^{\frac{1}{2}} \|u\|_{L^2}^{\frac{1}{2}} \|\Delta\phi\|_{L^4} \\
&\lesssim \delta(u(s))^{\frac{2}{3}} \delta(u(s))^{\frac{1}{3}} \|\Delta\phi\|_{L^4} \lesssim \delta(u(s)) \|\Delta\phi\|_{L^4}.
\end{align*}
 	
Let $f(z) = |z|^2 z$. Then, H\"older's inequality and Sobolev embedding imply:
\begin{align*}
|(f(u(s)) - f(W(s)), \phi)_{\dot{H}^1}| 
&= |(f(u(s)) - f(W(s)), \Delta\phi)_{L^2}| \\
&\leq \|u(s) - W(s)\|_{L^4} \left[\|u\|^2_{L^4} + \|W(s)\|^2_{L^4}\right] \|\Delta\phi\|_{L^4} \\
&\lesssim \delta(u(s)) \|\Delta\phi\|_{L^4}.
\end{align*}
 	
In particular, if $\phi = e^{-i\theta(s)} \mu(s) W\bigl(\mu(s)\cdot + x(s)\bigr)$ with $\Delta W \in L^4$, then 
\[
\|\Delta \phi\|_{L^4} = \mu(s)^{2} \|\Delta W\|_{L^4}
\]
and
\begin{align}\label{Dmy}
\bigl|\bigl(\tfrac{1}{\mu(s)^2} \bigl[(|W(s)|^2 W(s) - |u|^2 u) + |u| u \bigr], \phi\bigr)_{\dot{H}^1}\bigr| \lesssim \delta(u(s)).
\end{align}
 	
 Furthermore, we have the following estimates:
 \begin{align}\nonumber
 \bigl|(\tfrac{1}{\mu(s)^2} \Delta H, W(s))_{\dot{H}^1}\bigr| 
 &= \tfrac{1}{\mu(s)^2}|(H, \Delta W(s))_{\dot{H}^1}| \\\label{De111}
 &\leq \|H\|_{\dot{H}^1} \|\Delta W\|_{\dot{H}^1} \lesssim \delta(u(s)), \\\nonumber
 \bigl|(\tfrac{1}{\mu(s)^2} \Delta H, \Lambda W(s))_{\dot{H}^1}\bigr| 
 &= \tfrac{1}{\mu(s)^2}|(H, \Delta \Lambda W(s))_{\dot{H}^1}| \\\label{De222}
 &\leq  \|H\|_{\dot{H}^1} \|\Delta \Lambda W\|_{\dot{H}^1} \lesssim \delta(u(s)), \\ \nonumber
 \left|(\tfrac{1}{\mu(s)^2} \Delta H, [ \partial_{j} W](s))_{\dot{H}^1}\right| 
 &= \tfrac{1}{\mu(s)^2}|(H, [\Delta \partial_{j} W](s))_{\dot{H}^1}| \\    \label{De133}
 &\leq  \|H\|_{\dot{H}^1} \|[\Delta \partial_{j} W]\|_{\dot{H}^1} \lesssim \delta(u(s)),
\end{align}
for	$j=1,\dots,4$.	Now define 
\begin{align*} 
\epsilon(s) = \delta(u(s))\bigl[\delta(u(s)) + \bigl|\theta'(s)\bigr| + \bigl|\tfrac{\mu'(s)}{\mu(s)}\bigr| + \bigl|\tfrac{\mu'(s)}{\mu(s)} x(s) - x'(s)\bigr|\bigr].
\end{align*}
 	
Since $h(t) \in G^{\bot}$, we have $H(s) \perp \operatorname{span}\{W(s), [\nabla W](s), iW(s), [\Lambda W](s)\}$, which  implies:
\begin{align}\label{Casf}
\bigl| ( i\partial_s H, W(s) )_{\dot{H}^1} \bigr| + \bigl| ( i\partial_s H, [\tilde{\Lambda} W](s) )_{\dot{H}^1} \bigr| + \bigl| ( i\partial_s H, [\partial_{j} W](s) )_{\dot{H}^1} \bigr| \lesssim \epsilon(s),
 	\end{align}
 for	$j=1,\dots,4$.  Projecting \eqref{EwsE} in $\dot{H}^{1}$ onto $W(s)$, $[\partial_{j} W](s)$, $iW(s)$, $[\Lambda W](s)$ and using estimates \eqref{Dmy}--\eqref{Casf} and \eqref{mo1}, we infer that:
 \[
 \bigl|x'(t) - \tfrac{\mu'(t)}{\mu(t)}x(t)\bigr| + \bigl|\theta'(s)\bigr| + \bigl|\tfrac{\mu'(s)}{\mu(s)}\bigr| \lesssim \delta(u(s)) + \mathcal{O}(\epsilon(s)).
 \]
Choosing $\delta_0$ sufficiently small, we obtain \eqref{modu4}.\end{proof}
 
We now obtain Lemma~\ref{modula} directly from Lemmas~\ref{Lce1}, \ref{Ders}, and \ref{Molfa}.

\section{Failure of uniform spacetime bounds at the threshold}\label{sec:App}

The following result demonstrates that there can be no uniform estimate on the spacetime norms of solutions as one approaches the sharp energy threshold.  The basic idea is to use rescalings of the ground state solution for the energy-critical problem, which have growing spacetime norms and approximately solve \eqref{NLS} in the small-scale limit.  We additionally use frequency projections to ensure that the solutions belong to $H^1$ and decrease the amplitude slightly to ensure that the solutions remain below the energy threshold.

\begin{theorem}\label{T:nounif} There exists a sequence of global $H^1$ solutions $u_n:\R\times\R^4\to\C$ to \eqref{NLS} such that
\[
E(u_n)\nearrow E^c(W)\quad\text{and}\quad \|\nabla u_n(0)\|_{L^2} \nearrow \|\nabla W||_{L^2},
\]
with 
\[
\lim_{n\to\infty} \|\nabla u_n\|_{L_t^2 L_x^4(\R\times\R^4)} =\infty. 
\]
\end{theorem}

\begin{proof} Let $\eps_n,\lambda_n,N_n$ be positive monotonic sequences satisfying $\eps_n\to 0$, $\lambda_n\to\infty$, and $N_n\to0 $.

Define $\phi_n \in H^1$ by
\[
\phi_n(x) = (1-\eps_n) \lambda_n [P_{>N_n}W](\lambda_n x).
\]

By changing variables, one can readily verify that
\[
\|\nabla \phi_n\|_{L^2}\nearrow \|\nabla W\|_{L^2}.  
\]
Similarly, choosing $\lambda_n$ large enough that
\begin{equation}\label{enln}
\tfrac13 \lambda_n^{-\frac23} \|P_{> N_n}W\|_{L^3}^3 < \tfrac12\bigl[\tfrac{1-(1-\eps_n)^2}{(1-\eps_n)^3}\bigr]\|\nabla  P_{> N_n} W\|_{L^2} - \tfrac14\bigl[\tfrac{1-(1-\eps_n)^4}{(1-\eps_n)^3)}\bigr]\|P_{> N_n}W\|_{L^4}^4,
\end{equation}
we can arrange that
\[
E(\phi_n)\nearrow E^c(W). 
\] 
Here we rely on the fact that right-hand side of \eqref{enln} is positive for $\eps_n=\eps_n(W)$ and $N_n(W)$ sufficiently small.

We can therefore define global scattering solutions $u_n$ to \eqref{NLS} with $u_n(0)=\phi_n(x)$.  Fix $R\geq1$. Our goal is to use the stability result (Lemma~\ref{stabi}) to derive a lower bound on the norm of the solutions $u_n$ on the intervals $I_n=[-R\lambda_n^{-2},R\lambda_n^{-2}]$.

To this end, we define the approximate solutions
\[
\tilde u_n(t,x) \equiv \phi_n(x). 
\]

We first show that these approximate solutions obey good estimates (uniformly in $n$) on the intervals $I_n$.  We begin by changing variables and using Sobolev embedding to obtain
\[
\|\nabla \tilde u_n\|_{L_t^\infty L_x^2(I_n\times\R^4)} + \|\nabla \tilde u_n\|_{L_{t,x}^6(I_n\times\R^4)} \lesssim \|\nabla W\|_{L^2} + R^{\frac16}\|\nabla W\|_{L^6} \lesssim R^{\frac16}. 
\]
Note that we can use the explicit formula for $W$ to check the finiteness of these norms.

Next, by a change of variables and Bernstein estimates, 
\[
\| \tilde u_n\|_{L_t^\infty L_x^2(I_n\times\R^4)} \lesssim \lambda_n^{-1} N_n^{-1} \|\nabla W\|_{L^2} \lesssim 1
\]
provided we impose $\lambda_n \gtrsim N_n^{-1}$.  Moreover,
\[
\|\tilde u_n\|_{L_{t,x}^3(I_n\times\R^4)} \lesssim \lambda_n^{-1} R^{\frac13} \|W\|_{L_x^3} \lesssim \lambda_n^{-1} R^{\frac13}.
\]

Note also that
\begin{equation}\label{tildeunlb}
\| \nabla \tilde u_n \|_{L_t^2 L_x^4(I_n\times\R^4)} \sim R^{\frac12} \| \nabla P_{>N_n} W\|_{L_x^4}\gtrsim R^{\frac12} 
\end{equation}
uniformly in $n$. 

We now define
\[
e_n = (i\partial_t + \Delta) \tilde u_n + |\tilde u_n|^2 \tilde u_n - |\tilde u_n| \tilde u_n. 
\]
To find an explicit expression for $e_n$, we use the equation $-\Delta W = W^3$.  This leads to
\[
e_n = e_{n,1}+e_{n,2}+e_{n,3},
\]
where
\begin{align*}
e_{n,1} &= \lambda_n^3\{(1-\eps_n)^3-(1-\eps_n)\}[(P_{> N_n}W)^3](\lambda_n x), \\
e_{n,2} & =\lambda_n^3[1-\eps_n] \{ (P_{> N_n} W)^3(\lambda_n x) - [P_{> N_n}(W^3)](\lambda_n x)\}, \\
e_{n,3} & = -\lambda_n^2(1-\eps_n)^2 [(P_{> N_n}W)^2](\lambda_n x).
\end{align*}
As $W\in L^{\frac{9}{2}}$, we can change variables and use H\"older's inequality to obtain
\[
\| (\nabla)^s e_{n,1}\|_{L_{t,x}^{\frac32}(I_n\times\R^4)}\lesssim \eps_n R^{\frac23}\lambda_n^{s-1}\|W\|_{L^{\frac{9}{2}}}^3 \lesssim \eps_n R^{\frac23}\lambda_n^{s-1} \to 0 \qtq{as}n\to\infty
\]
for $s\in\{0,1\}$. Similarly, as $W\in L^3$, 
\[
\| (\nabla)^s e_{n,3}\|_{L_{t,x}^{\frac32}(I_n\times\R^4)} \lesssim R^{\frac23} \lambda_n^{s-2}\|W\|_{L^3}^2 \lesssim R^{\frac23}\lambda_n^{s-2}\to 0 \qtq{as}n\to\infty
\]
for $s\in\{0,1\}$.

It remains to estimate $e_{n,2}$. We begin by observing that
\begin{align}
\|(\nabla)^s &e_{n,2}\|_{L_{t,x}^{\frac32}(I_n\times\R^4)} \nonumber\\
&\lesssim R^{\frac23} \lambda_n^{s-1}\| (\nabla)^s[(P_{> N_n} W)^3 - P_{> N_n}(W^3)\|_{L_x^{\frac32}} \nonumber\\
& \lesssim R^{\frac23}\lambda_n^{s-1}\bigl\{ \|(\nabla)^s P_{\leq N_n}(W^3)\|_{L_x^{\frac32}} + \|(\nabla)^s\{ (P_{> N_n} W)^3 - W^3\}\|_{L_x^{\frac32}}\bigr\}.\label{jason1}
\end{align}

We estimate each norm separately.  For the first term in \eqref{jason1} with $s=0$, we have
\[
R^{\frac23}\lambda_n^{-1} \|P_{\leq N_n}(W^3)\|_{L^{\frac32}} \lesssim R^{\frac23}\lambda_n^{-1} \|W\|_{L^{\frac92}}^3 \to 0 \qtq{as}n\to\infty.
\]
For $s=1$, we instead use Bernstein estimates, the product rule, H\"older's inequality, and Sobolev embedding to obtain
\begin{align*}
R^{\frac23}\| \nabla P_{\leq N_n}(W^3)\|_{L_x^{\frac32}} & \lesssim R^{\frac23} N_n \|W\|_{L_x^{\frac92}}^3\to 0 \qtq{as}n\to\infty.
\end{align*}

Similarly, for the second term in \eqref{jason1} with $s=0$ we have
\[
R^{\frac23} \lambda_n^{-1} \| (P_{> N_n} W)^3 - W^3\|_{L_x^{\frac32}} \lesssim R^{\frac23}\lambda_n^{-1} \|W\|_{L^{\frac92}}^3 \to 0\qtq{as}n\to\infty.
\]
For $s=1$, we use the schematic notation
\[
(P_{> N_n}W)^3 - W^3 = V^2 P_{\leq N_n}W,\quad V\in\{P_{\leq N_n} W,P_{> N_n}W\},
\]
so that
\[
\nabla[(P_{> N_n}W)^3 - W^3] = V\nabla V P_{\leq N_n}W + V^2 \nabla P_{\leq N_n}W. 
\]
With this notation, we use Bernstein estimates to obtain
\begin{align*}
R^{\frac23}& \| \nabla\{(P_{> N_n} W)^3 - W^3\}\|_{L_x^{\frac32}} \\
& \lesssim R^{\frac23} \|V\|_{L^\infty}\{\|\nabla V\|_{L^2} \|P_{\leq N_n} W\|_{L^6} + \|V\|_{L^3} \|\nabla P_{\leq N_n} W\|_{L^3}\} \\
& \lesssim R^{\frac23}N_n^{\frac23} \|W||_{L^\infty}\|\nabla W\|_{L^2} \|W\|_{L^3}.
\end{align*}

Collecting the estimates above, we obtain
\[
\| \langle\nabla\rangle e_n\|_{L_{t,x}^{\frac32}(I_n\times\R^4)}\lesssim R^{\frac23}\{\eps_n+\lambda_n^{-1}+N_n^{\frac23}\}\to 0 \qtq{as}n\to\infty.
\]

Thus, by Lemma~\ref{stabi}  we derive that for $n=n(R)$ sufficiently large, 
\[
\|\nabla[u_n-\tilde u_n]\|_{L_t^2 L_x^4(I_n\times\R^4)} \ll \tfrac{1}{10}R^{\frac12},
\]
so that (recalling \eqref{tildeunlb})
\[
\|\nabla u_n\|_{L_t^2 L_x^4(\R\times\R^4)} \geq \|\nabla u_n\|_{L_t^2 L_x^4(I_n\times\R^4)}\gtrsim R^{\frac12}. 
\]
As $R$ was arbitrary, we conclude that
\[
\lim_{n\to\infty} \|\nabla u_n\|_{L_t^2 L_x^4(\R\times\R^4)} = \infty,
\]
as desired. 
\end{proof}


\end{document}